\documentclass[a4paper,leqno,12pt, reqno]{amsart}

\usepackage{amsmath}
\usepackage{amssymb}
\usepackage{graphicx}
\usepackage{tikz-cd}
\usepackage{hyperref}
\usepackage[all]{xy}

\usepackage{xspace}
\usepackage{bm}
\usepackage{amsmath}
\usepackage{amstext}
\usepackage{amsfonts}
\usepackage[mathscr]{euscript}
\usepackage{amscd}
\usepackage{latexsym}
\usepackage{amssymb}
\usepackage{enumerate}
\usepackage{xcolor}
\usepackage{mathtools}

\theoremstyle{plain}
\swapnumbers
    \newtheorem{theorem}{Theorem}[section]
    \newtheorem{proposition}[theorem]{Proposition}
    \newtheorem{lemma}[theorem]{Lemma}

    \newtheorem{corollary}[theorem]{Corollary}
    
    \newtheorem{subsec}[theorem]{}

\theoremstyle{definition}
    \newtheorem{definition}[theorem]{Definition}
    \newtheorem{example}[theorem]{Example}

\theoremstyle{remark}
        \newtheorem{remark}[theorem]{Remark}

    \newtheorem{ack}[theorem]{Acknowledgements}

\newcommand{\sect}{\setcounter{figure}{0}\section}

\renewcommand{\thefigure}{\arabic{section}.\arabic{figure}}

\newenvironment{myeq}[1][]
{\stepcounter{figure}\begin{equation}\tag{\thefigure}{#1}}
{\end{equation}}

\newcommand{\Aut}{\operatorname{Aut}}
\newcommand{\conn}{\operatorname{conn}}

\newcommand{\End}{\operatorname{End}}

\newcommand{\Hom}{\operatorname{Hom}}
\newcommand{\Id}{\operatorname{Id}}

\newcommand{\cp}{\operatorname{cup}}

\newcommand{\A}{\mathcal{A}_{\#}}
\renewcommand{\AA}{\mathcal{A}_{\ast}}
\newcommand{\AAA}{\mathcal{A}^{\ast}}
\newcommand{\NA}{N\A}
\newcommand{\NAA}{N\AA}
\newcommand{\NAAA}{N\AAA}
\newcommand{\R}{\mathbb{R}P}
\newcommand{\C}{\mathbb{C}P}
\renewcommand{\H}{\mathbb{H}P}

\newcommand{\E}{\mathbb{E}P}
\newcommand{\ZZ}{{\mathbb Z}}

\begin{document}

\title{Reducibility of self-maps in monoid and its related invariants }
\author{Gopal Chandra Dutta}

\address{Stat-Math Unit, Indian Statistical Institute, Kolkata 700108, India}
\email{gdutta670@gmail.com}

\date{\today}


\begin{abstract}
Given a positive integer $k$, we investigate the $k$-redcibility of self-maps in the monoid $\AA^k(X\vee Y)$, consisting of self-maps that induce isomorphisms on homology groups up to degree $k$. In general, verifying $k$-reducibility is a subtle problem. We show that the $k$-reducibility of a self-map is determine through its induced endomorphisms on homology or cohomology groups. Moreover, under the k-reducibility assumption, the computation of the homology self-closeness number of the wedge sum of spaces essentially reduces to the computation of the homology self-closeness numbers of the individual wedge summands. We generalize the notion of an atomic space to that of an $n$-atomic space and establish some of its fundamental properties.  We show that the $k$-reducibility criteria for self-maps in a monoid $\AA^k(X)$ is satisfied when the space $X$ decomposes as a wedge sum of distinct $n$-atomic spaces. Finally, we determine the homology self-closeness numbers of wedge sums of distinct $n$-atomic spaces.
\end{abstract}
\maketitle

\maketitle

\sect{Introduction}
For a pointed space $X$, the group of homotopy classes of self-equivalences of $X$ is denoted by $\Aut(X)$. This group has been studied by several authors over the last few decades \cite{AGSE, RGSHD}. In general, however, determining the structure of $\Aut(X)$ is a difficult problem. In \cite{CLCNGS}, Choi and Lee introduced a monoid $\A^k(X)$ closely related to the group $\Aut(X)$, defined by $$\A^k(X): = \big\{f\in [X,X]: ~f_{\#}\colon \pi_i(X)\xrightarrow{ \cong } \pi_i(X), \text{ for all } 0\leq i\leq k \big\}.$$ Later, this notion was extended to homological and cohomological settings by Oda and Yamaguchi in \cite{OYSCFC}. They defined $$\AA^k(X): = \big\{f\in [X,X]: ~f_{\ast}\colon H_i(X)\xrightarrow{ \cong } H_i(X), \text{ for all } 0\leq i\leq k \big\}.$$ Similarly, in cohomology, $\AAA_k(X): = \big\{f\in [X,X]: ~f^{\ast}\colon H^i(X)\xrightarrow{ \cong } H^i(X), 0\leq i\leq k \big\}$. 

Let $f\colon X\times Y\to X\times Y$ be a self-map. It admits a decomposition into four component maps $f_{IJ}\colon J\to I$, defined as $f_{IJ}:= p'_I\circ f\circ \iota'_J$, where $\iota'_J\colon J\to X\times Y$ and $p'_I\colon X\times Y\to I$ are coordinate inclusion and projection maps, respectively for $I,J\in \{X,Y\}$. A map $f\in \Aut(X\times Y)$ is said to be \emph{reducible} if $f_{XX}\in \Aut(X)$ and $f_{YY}\in \Aut(Y)$. The structure of $\Aut(X\times Y)$ under this reducibility assumption has been studied in \cite{PPSHE,PPRSE}. An analogous notion of reducibility exists for wedge sums, and the group $\Aut(X\vee Y)$ has been investigated under similar hypotheses in \cite{HWSW}. Moreover, for a fix positive integer $k$, a map $f\in \A^k(X\times Y)$ is called \emph{$k$-reducible} if $f_{XX}\in \A^k(X)$ and $f_{YY}\in \A^k(Y)$. Under this $k$-reducibility assumption, the monoid $\A^k(X\times Y)$ decomposes as the product of the individual monoids $\A^k(X)$ and $\A^k(Y)$, as shown in \cite{DSMSE, JLFSMP}. Similarly, a self-map $f\in \AA^k(X\times Y)$ is said to be a $k$-reducible if $f_{XX}\in \AA^k(X)$ and $f_{YY}\in \AA^k(Y)$. Nevertheless, reducibility is a rather restrictive condition and does not hold in general. In this paper, we establish several criteria ensuring $k$-reducibility of self-maps in $\AA^k(X\vee Y)$, mostly generalizing the results of Pave\v si\'c in \cite{PPRSE}. The notion of $k$-reducibility plays a crucial role in understanding the structure of the monoid $\AA^k(X\vee Y)$ and, consequently, the group $\Aut(X\vee Y)$.

For a connected CW-complex $X$, Choi and Lee introduced a numeric homotopy invariant called \emph{(homotopy) self-closeness number}, defined by $$\NA(X): = \min \big\{k\geq 0: ~\A^k(X) = \Aut(X) \big\}.$$ Similarly, for a simply connected CW–complex, Oda and Yamaguchi defined the homology and cohomology self-closeness numbers as $$\NAA(X): = \min \big\{k\geq 0: ~\AA^k(X) = \Aut(X) \big\},$$ $$\NAAA(X): = \min \big\{k\geq 0: ~\AAA_k(X) = \Aut(X) \big\}.$$
These numbers have been computed by many authors over the last one decade in \cite{LCHSCS, OYSEC, OYSF, OYSCFC}. The (homotopy) self-closeness number of the product of spaces has been studied under the $k$-reducibility assumptions in \cite{DSMSE, LSPS}. We establish some results on the homology self-closeness number of wedge sums using the $k$-reducibility assumption. Also, we directly compute the homology self-closeness numbers of wedge sums of certain spaces for which the $k$-reducibility assumption does not hold.  

The paper is organized as follows. In Section \ref{reducibility}, we explicitly formulate the notion of $k$-reducibility for self-maps in the monoid $\AA^k(X\vee Y)$. First, we establish a relation between the monoids $\AA^k(X\times Y)$ and $\AA^k(X\vee Y)$ and this is followed by a relation of the $k$-reducibility of self-maps between the monoids $\AA^k(X\vee Y)$ and $\AA^k(X\times Y)$.  We derive explicit connections between the $k$-reducibility of self-maps in the monoid associated with wedge sums and that of their suspensions, as well as between the monoid of product spaces and that of their loop spaces; see Proposition \ref{prop:redu_suspension} and \ref{prop:redu_loop}, respectively. Further, we determine some sufficient conditions for $k$-reducibility of self-maps in $\AA^k(X\vee Y)$; see Lemma \ref{red}. Since cohomology rings $H^*(X)$ and $H^*(Y)$ have a richer algebraic structure, ring homomorphisms induced in cohomology by self-maps of $X\vee Y$ provide an effective tool to verify the $k$-reducibility of self-maps in $\AAA_k(X\vee Y)$ and $\AA^k(X\vee Y)$; see Theorem \ref{thm:redu_cohom}. We find some useful connections between the $k$-reducibility of a self-map and the nilpotency of its induced endomorphisms on homology and cohomology groups in Section \ref{nilpotent_reducibility}. In particular, we obtain a sufficient condition for the $k$-reducibility of self-maps in $\AA^k(X\vee Y)$; see Theorem \ref{redu_radical}. Further, we establish a sufficient condition for the nilpotency of an endomorphism over a finite abelian group in Lemma \ref{prime_nilpotent}, which leads to a $k$-reducibility result in Proposition \ref{nilpotent_reducibility_homdis}. The nilpotent ring endomorphisms of the cohomology rings $H^*(X)$ and $H^*(Y)$ are shown to determine the reducibility of self-maps in the group $\Aut(X\vee Y)$; see Theorem \ref{thm_redu_multi_generators}. In Section \ref{self-closeness number}, we compute the homology self-closeness numbers of the wedge of spaces under the assumption of $k$-reducibility. We determine the homology self-closeness number of the wedge sum of two spaces in terms of the corresponding invariants of the individual spaces; see Theorem \ref{equality_self-closeness_wedge_product}. We also investigate several relations between the homology self-closeness numbers of products and wedge sums of spaces. Moreover, we employ a cohomological approach to compute these numbers in certain cases where the $k$-reducibility assumption of a self-map in the monoid of wedge sums fails; see Proposition \ref{wed}. In Section \ref{atomic decomposition}, we introduce the notion of an $n$-atomic space. We explore relations among the atomicity of a space, its suspension and its loop space in Proposition \ref{atomicity_loop_suspension}. Spaces of this type play a crucial role in determining the k-reducibility of self-maps in the monoid $\AA^k(X)$, where the space $X$ admits decompositions as wedges of atomic spaces; see Theorem \ref{reducibility_atomic}. Finally, in Section \ref{application} we present applications of the preceding results, with primary emphasis on the homology self-closeness number of an $n$-atomic space; see Lemma \ref{atomic_selfcloseness}. By analyzing $k$-reducibility, we determine the homology self-closeness numbers of wedge sums of finitely many 
$n$-atomic spaces whose individual components possess distinct homology self-closeness numbers; see Theorem \ref{thm_selfcloseness_redu} and Theorem \ref{thm_higher_selfcloseness_redu}.

In this paper, all spaces have the homotopy type of simply connected CW-complexes, and all maps are basepoint-preserving. Moreover, for notational convenience, we do not distinguish between a map $f\colon X\to X$ and its homotopy class in $[X,X]$. Throughout the paper, all homology groups are considered in the reduced sense. For an abelian group $G$, the set of all endomorphisms of $G$ is denoted by $\End(G)$. This is a ring whose set of units is $\Aut(G)$, the collection of all automorphisms of $G$. 
\begin{ack}
I would like to thank the National Board for Higher Mathematics (NBHM) for Grant No. 0204/7/2025/R\&D-II/8842.
\end{ack}

\sect{Reducibility}\label{reducibility}
The $k$-reducibility of a self-map in $\AA^k(X\vee Y)$ (or $\AA^k(X\times Y)$) is, in general, a subtle property that depends strongly on the structure of the underlying based spaces $X$ and $Y$. This notion plays a significant role in the determination of the group of self-homotopy equivalences of wedge sums and product spaces in \cite{HWSW, PPSHE, PPRSE}. In this section, we primarily investigate the $k$-reducibility of self-maps in the monoid $\AA^k(X\vee Y)$. Under the assumption of $k$-reducibility, we establish several results that are applied in the subsequent sections. To carry out this study, we first recall some basic definitions and conventions..  The \emph{connectivity} of a space $X$ is defined by $\conn(X):= \min \big\{i\geq 0~:~ \pi_{i+1}(X)\neq 0\big\}$. This implies $\pi_i(X) = 0$ for all $i\leq \conn(X)$. 

A self-map $f\colon X\to X'$ is called a \emph{homotopical $n$-equivalence} (or \emph{homological $n$-equivalence}) if the induced map $f_{\#}\colon \pi_i(X)\to \pi_i(X')$ (or $f_{\ast}\colon H_i(X)\to H_i(X')$) is an isomorphism for all $i<n$ and an epimorphism for $i=n$. Similarly, a self-map $f\colon X\to X'$ is called a \emph{cohomological} $n$-equivalence if $f^{\ast}\colon H^i(X)\to H^i(X')$ is an isomorphism for all $i<n$ and a monomorphism for $i=n$; see \cite[Section 6.4]{AIH}.

\begin{lemma}[{\cite[Lemma 6.4.13, Theorem 6.4.15]{AIH}}]\label{leq}
Let $f\colon X\to X'$ be a map between simply connected spaces. Then the following are equivalent:
\begin{enumerate}[{(a)}]
\item $f$ is a homotopical $n$-equivalence.
\item $f$ is a homological $n$-equivalence.
\item $f$ is a cohomological $n$-equivalence.
\end{enumerate}
\end{lemma}

%\begin{proof}
%By Lemma 6.4.11, Proposition 6.4.14 and Theorem 6.4.15 of \cite{AIH}.
%\end{proof}

A self-map $g\colon X\vee Y\to X\vee Y$ determines two component maps $g\circ \iota_X\colon X\to X\vee Y$ and $g\circ \iota_Y\colon Y\to X\vee Y$. For each pair $I,J\in \{X,Y\}$, we defined a map $g_{IJ}\colon J\to I$ by $g_{IJ}:= p_I\circ g\circ \iota_J$. Here $\iota_J\colon J\to X\vee Y$ and $p_I\colon X\vee Y\to I$ are coordinate inclusion and projection maps, respectively. 
 
\begin{definition}\label{def:hom_redu}
A self-map $g\in \AA^k(X\vee Y)$ is said to be a \emph{$k$-reducible} if $g_{XX}\in \AA^k(X)$ and $g_{YY}\in \AA^k(Y)$.
\end{definition}

\begin{definition}\label{def:cohom_redu}
A self-map $g\in \AAA_k(X\vee Y)$ is said to be a \emph{$k$-reducible} if $g_{XX}\in \AAA_k(X)$ and $g_{YY}\in \AAA_k(Y)$.
\end{definition}

Note that a self-map $g\in \Aut(X\vee Y)$ is said to be reducible if it is $k$-reducible for every $k\geq 1$, either in the homological or in the cohomological sense, as defined in Definition \ref{def:hom_redu} and \ref{def:cohom_redu}. However, the converse statement does not necessarily hold. In particular, even in situations where every self-map in the group $\Aut(X\vee Y)$ is reducible, it may still be difficult to verify the $k$-reducibility of self-maps in the monoids $\AA^k(X\vee Y)$ and $\AAA_k(X\vee Y)$; see Theorem \ref{thm_redu_multi_generators}.

\begin{lemma}\label{redu_wedge_product}
Suppose $f$ is a self-map in $\AA^k(X\times Y)$ such that $g := f|_{X\vee Y}$ is in $\AA^k(X\vee Y)$. Then $f$ is $k$-reducible if and only if $g$ is $k$-reducible.
\end{lemma}

\begin{proof}
A self-map $f\colon X\times Y\to X\times Y$ can be expressed as $f = (f_X,f_Y)$, where $f_X = p'_X\circ f\colon X\times Y\to X$ and $f_Y = p'_Y\circ f\colon X\times Y\to Y$. Moreover, a self-map $g\colon X\vee Y\to X\vee Y$ is represented as $g = (g_X,g_Y)$, where $g_X = g\circ \iota_X\colon X\to X\vee Y$ and $g_Y = g\circ \iota_Y\colon Y\to X\vee Y$.
Therefore, for $I,J\in \{X,Y\}$, we have $f_{IJ} = p'_I\circ f\circ \iota'_J\colon J\to I$, where $p'_I\colon X\times Y\to I$ and $\iota'_J \colon J\to X\times Y$.
Similarly, for $I,J\in \{X,Y\}$, we have $g_{IJ} = p_I\circ g\circ \iota_J \colon J\to I$, where $p_I\colon X\vee Y\to I$ and  $\iota_J\colon J\to X\vee Y$.

Observe that $g_{XX} = p_X\circ g\circ \iota_X = p'_X\circ f\circ \iota'_X = f_{XX}$ and
$g_{YY} = p_Y\circ g\circ \iota_Y = p'_Y\circ f\circ \iota'_Y = f_{YY}$. This completes the proof.
\end{proof}

\begin{lemma}\label{homo_monoid}
Assume that $X$ and $Y$ are finite-dimensional CW-complexes satisfying the condition $\conn(X\times Y, X\vee Y)\geq \dim(X\vee Y)$. Suppose $H_m(X)$ and $H_m(Y)$ are finitely generated, where $m = \dim(X\vee Y)$.
\begin{enumerate}[(a)]
\item For each $k$, we obtain a monoid homomorphism $$\phi\colon \AA^k(X\times Y)\to \AA^k(X\vee Y), \text{ defined as } \phi(f):= f|_{X\vee Y}.$$
\item For each $k$, there exists a monoid homomorphism $$\Phi \colon \AA^k(X\times Y)\to \AA^k(X\wedge Y), ~\text{ defined by }  \Phi(f):= \bar{f}\colon X\times Y/_{X\vee Y}\to X\times Y/_{X\vee Y}.$$
 \end{enumerate}
\end{lemma}

\begin{proof}
Let $f\colon X\times Y\to X\times Y$ be a cellular map. Since $\conn(X\times Y, X\vee Y)\geq \dim(X\vee Y)$, we have $f|_{X\vee Y}\colon X\vee Y\to X\vee Y$. It follows that there exists a map over the quotient $\bar{f}\colon X\wedge Y=X\to X\wedge Y$. See the following commutative diagram:
\begin{equation}\label{diag_cellular}
\xymatrixrowsep{2 em}
\xymatrixcolsep{5 em}
\xymatrix{
X\wedge Y  \ar[r]^{\bar{f}}  & X \wedge Y \\
X\times Y \ar[u]^{q} \ar[r]^{f} & X\times Y \ar[u]_{q}\\
X\vee Y \ar[u]^{\iota} \ar[r]_{f|_{X\vee Y}} & X\vee Y \ar[u]_{\iota}
}
\end{equation}
\begin{enumerate}[(a)]\setlength\itemsep{1.5 em}
\item Any self-map from $X\times Y$ to itself can be replace by a cellular map over $X\times Y$ to itself upto homotopy, by cellular approximation theorem. Thus, $(f\circ g)|_{X\vee Y} = f|_{X\vee Y}\circ g|_{X\vee Y}$. Therefore, $\phi(f\circ g) = \phi(f)\circ \phi(g)$. Moreover, $\iota\colon X\vee Y\to X\times Y$ is an $m$-equivalence, where $m = \dim(X\vee Y)$. If $k<m$ and $f\in \AA^k(X\times Y)$, then the commutativity of the homology groups associated with diagram \eqref{diag_cellular} implies $f|_{X\vee Y}\in \AA^k(X\vee Y)$. It is sufficient to show for $k=m$. If $f\in \AA^m(X\times Y)$, then $(f|_{X\vee Y})_*\colon H_m(X\vee Y)\to H_m(X\vee Y)$ is an epimorphism and so an isomorphism, since $H_m(X\vee Y)$ is fintely generated. Therefore, $\phi$ is a homomorphism of monoids.
\item Let $f\in \AA^k(X\times Y)$. From part (a), we have $f|_{X\vee Y}\in \AA^k(X\vee Y)$. Applying the five lemma on the following commutative diagram of abelian groups:
\begin{equation}\label{diag_five}
\xymatrixrowsep{3em}
\xymatrixcolsep{2em}
\xymatrix{
	H_i(X\vee Y)   \ar[r]^{\iota_*}  \ar[d]^{(f|{X\vee Y})_*} &  H_i(X\times Y) \ar[d]^{f_*} \ar[r]^{q_*} & H_i(X\wedge Y) \ar[d]^{\bar{f}_*} \ar[r]^{} & H_{i-1}(X\vee Y) \ar[d]^{(f|_{X\vee Y})_*} \ar[r]^{\iota_*} & H_{i-1}(X\times Y) \ar[d]^{f_*} \\
	H_i(X\vee Y) \ar[r]_{\iota_*} & H_i(X\times Y) \ar[r]_{q_*} & H_i(X\wedge Y) \ar[r]_{} & H_{i-1}(X\vee Y)  \ar[r]_{\iota_*} & H_{i-1}(X\times Y).
	}
\end{equation}
It follows that $\bar{f}\in \AA^k(X\wedge Y).$ Further, $\Phi(f\circ g) = \overline{f\circ g} = \bar{f}\circ \bar{g} = \Phi(f)\circ \Phi(g)$. Therefore, we get the desired result.
\end{enumerate}
\end{proof}

The following corollary follows immediately from Lemma \ref{redu_wedge_product} and Lemma \ref{homo_monoid}.
\begin{corollary}\label{redu_monoid_wedge_product}
Suppose $X$ and $Y$ are finite-dimensional CW-complexes such that $\conn(X\times Y, X\vee Y)\geq \dim(X\vee Y)$. Let $H_m(X)$ and $H_m(Y)$ be finitely generated, where $m=\dim(X\vee Y)$. For each fixed $k\geq 1$, if every self-map in $\AA^k(X\vee Y)$ is $k$-reducible, then any self-map in $\AA^k(X\times Y)$ is $k$-reducible.
\end{corollary}

Recall the generalized Freudenthal suspension theorem: Let $Y$ be an $n$-connected CW-complex and $X$ be a finite dimensional CW-complex. Then the suspension map $\Sigma\colon [X,Y]\to [\Sigma X,\Sigma Y]$ is a bijection if $\dim(X)< 2n+1$ and a surjection if $\dim(X) = 2n+1$ (\cite[Theorem 1.21]{CSH}).

The following proposition establishes a relationship between the
$k$-reducibility of self-maps on a wedge of spaces and on the wedge of their suspension spaces.
\begin{proposition}\label{prop:redu_suspension}
Fix a $k\geq 1$, any self-map in $\AA^k(X\vee Y)$ is $k$-reducible if every self-map in $\AA^{k+m}(\Sigma^m X\vee \Sigma^m Y)$ is $(k+m)$-reducible for some $m\geq 1$. 

Moreover, if $X\vee Y$ is $n$-connected and $\dim(X\vee Y)\leq 2n+1$, then $k$-reducibility of each self-map in $\AA^k(X\vee Y)$ implies $(k+m)$-reducibility of every self-maps in $\AA^{k+m}(\Sigma^m X\vee \Sigma^m Y)$ for all $m\geq 1$.
\end{proposition}

\begin{proof}
For $m\geq 1$, we know that $\Sigma^m(X\vee Y)\simeq \Sigma^m X\vee \Sigma^m Y$. Let $f\colon X\to X$ be a self-map. Consider a commutative diagram:
\begin{equation}\label{redu_suspension}
\xymatrixrowsep{2 em}
\xymatrixcolsep{5 em}
\xymatrix{
H_i(X) \ar[r]^{f_*} \ar[d]_{\cong}   & H_i(X) \ar[d]^{\cong}\\
H_{i+m}(\Sigma^m X) \ar[r]_{(\Sigma^m f)_*} &  H_{i+m}(\Sigma^m X).
}
\end{equation}
Thus, for $m\geq 1$, we have $f\in \AA^k(X)$ if and only if $\Sigma^m f\in \AA^{k+m}(\Sigma^m X)$. Let $g\in \AA^k(X\vee Y)$. Then $\Sigma^m g\in \AA^{k+m}(\Sigma^m X\vee \Sigma^m Y)$. By the given assumption, we obtain $(\Sigma^m g)_{\Sigma^m X \Sigma^m X}\in \AA^{k+m}(\Sigma^m X)$ and $(\Sigma^m g)_{\Sigma^m Y \Sigma^m Y}\in \AA^{k+m}(\Sigma^m Y)$ for some $m\geq 1$. This implies $\Sigma^m g_{XX} = \Sigma^m (p_X\circ g\circ \iota_X) = p_{\Sigma^m X}\circ \Sigma^m g\circ \iota_{\Sigma^m X} = (\Sigma^m g)_{\Sigma^m X \Sigma^m X}\in \AA^{k+m}(\Sigma^m X)$ and $\Sigma^m g_{YY} = \Sigma^m (p_Y\circ g\circ \iota_Y) = p_{\Sigma^m Y}\circ \Sigma^m g\circ \iota_{\Sigma^m Y} = (\Sigma^m g)_{\Sigma^m Y \Sigma^m Y}\in \AA^{k+m}(\Sigma^m Y)$. It follows that $g_{XX}\in \AA^k(X)$ and $g_{YY}\in \AA^k(Y)$. For the converse part, observe that $[X\vee Y,X\vee Y]\xrightarrow{\Sigma^m} [\Sigma^m X\vee \Sigma^m Y, \Sigma^m X\vee \Sigma^m Y]$ is surjective for all $m\geq 1$. Let $h\in \AA^{k+m}(\Sigma^m X\vee \Sigma^m Y)$. Then there exists a map $h'\colon X\vee Y\to X\vee Y$ such that $\Sigma^m h' = h$. It follows that $h'\in \AA^k(X\vee Y)$. Using an argument similar to the above, we have $h_{\Sigma^m X \Sigma^m X}\in \AA^{k+m}(\Sigma^m X)$ and $h_{\Sigma^m Y \Sigma^m Y}\in \AA^{k+m}(\Sigma^m Y)$ for all $m\geq 1$. This completes the proof.
\end{proof}

Recall the facts that $\Omega^m(X\times Y)\simeq \Omega^m X\times \Omega^m Y$ and $\pi_{i+m}(X)\xrightarrow{\cong} \pi_i(\Omega^m X)$ for all $m\geq 1$. Fix an integer $k\geq 1$. For any self-map $\phi\colon X\to X$, it follows that $\phi\in \A^k(X)$ if and only if $\Omega^m \phi\in \A^{k-m}(\Omega^m X)$ for all $m\geq 1$. Therefore, we obtain the following proposition. The proof is similar to that of Proposition \ref{prop:redu_suspension} and is therefore omitted. 
\begin{proposition}\label{prop:redu_loop}
Fix a $k\geq 1$, any self-map in $\A^k(X\times Y)$ is $k$-reducible if every self-maps in $\A^{k-m}(\Omega^m X\times \Omega^m Y)$ is $(k-m)$-reducible for some $1\leq m<k$.
\end{proposition}

\begin{corollary}
Suppose $X$ and $Y$ are CW-complexes of finite types. Fix an integer $k\geq 1$, every self-maps in $\AA^{k-m}(\Omega^m X\times \Omega^m Y)$ is $(k-m)$-reducible and $\Omega^m X, \Omega^m Y$ are simply connected for some $1\leq m <k$. Then each self-map in $\AA^k(X\times Y)$ is $k$-reducible.
\end{corollary}

\begin{proof}
Let $f\in \AA^k(X\times Y)$. By Lemma \ref{leq}, the map $f\colon X\times Y\to X\times Y$ is a homotopical $k$-equivalence. Moreover, $\pi_k(X\times Y)$ is finitely  generated abelian group. It follows that $f\in \A^k(X\times Y)$ using the fact that $\pi_k(X\times Y)$ is a Hopfian group. This implies $\Omega^m f\in \A^{k-m}(\Omega^m X\times \Omega^m Y)$. Further, applying Lemma \ref{leq} and using the fact that $H_{k-m}(\Omega^m X\times \Omega^m Y)$ is a finitely generated abelian group, we conclude that $\Omega^m f\in \AA^{k-m}(\Omega^m X\times \Omega^m Y)$. Therefore, by the given assumption and repeated application of Lemma \ref{leq}, the desired result follows. 
\end{proof}

The following lemma establishes some sufficient conditions for $k$-reducibility.
\begin{lemma}\label{red}
Each self-map in the monoid $\AA^k(X\vee Y)$ is $k$-reducible if any one of the following conditions holds:
\begin{enumerate}[(i)]
\item For each $i\leq k$, at least one of $\Hom(H_i(X), H_i(Y)) $ or $\Hom(H_i(Y), H_i(X))$ is trivial.

\item For each $i\leq k$, at least one of $\Hom(H^i(X), H^i(Y)) $ or $\Hom(H^i(Y), H^i(X))$ is trivial and $H_k(X\vee Y)$ is finitely generated.

\item For each $i\leq k$, at least one of the induced homomorphisms $H^i(f)$ or $H^i(g)$, arising from maps $f\colon X\to Y$ and $g\colon Y\to X$, is trivial. In addition, $H_k(X\vee Y)$ is finitely generated.
\item For each $i\leq k$, both the groups $H_i(X), H_i(Y)$ are finite and have no common direct factors.
\end{enumerate}
\end{lemma}

\begin{proof}
Let $f\colon X\vee Y\to X\vee Y$ be a self-map in $\AA^k(X\vee Y)$. The induced endomorphisms on homology and cohomology, $f_*\colon H_i(X)\oplus H_i(Y)\to H_i(X)\oplus H_i(Y)$ and $f^*\colon H^i(X)\oplus H^i(Y)\to H^i(X)\oplus H^i(Y)$ admit the following matrix representations:

\[
M_{i}(f):=
\begin{bmatrix}
(f_{XX})_{*} & (f_{XY})_{*} \\
(f_{YX})_{*} & (f_{YY})_{*} 
\end{bmatrix}
\]
and 
\[
M^{i}(f):=
\begin{bmatrix}
(f_{XX})^{*} & (f_{YX})^{*} \\
(f_{XY})^{*} & (f_{YY})^{*} 
\end{bmatrix}.
\]
It follows that $f_*$ (respectively, $f^*$) is an isomorphism if and only if the corresponding matrix $M_i(f)$ (respectively, $M^i(f)$) is invertible. For details, see \cite[Section 2]{PPRSE}.
\begin{enumerate}[(i)]\setlength\itemsep{1em}
\item For each $i\leq k$, the matrix $M_i(f)$ is either an upper triangular or a lower triangular by the given assumption. Thus $M_i(f)$ is an invertible matrix if and only if $(f_{XX})_{*}\colon H_i(X)\to H_i(X)$ and $(f_{YY})_{*}\colon H_i(Y)\to H_i(Y)$ are isomorphisms. Hence $f_{XX}\in \AA^k(X)$ and $f_{YY}\in \AA^k(Y)$.

\item Let $f\in \AA^k(X\vee Y)$. By universal coefficient theorem of cohomology, we have $f\in \AAA_k(X\vee Y)$. Similar to part (i), it follows that $f_{XX}\in \AAA_k(X)$ and $f_{YY}\in \AAA_k(Y)$. Thus the self-maps $f_{XX}$ and $f_{YY}$ are cohomological $k$-equivalence, and hence so homological $k$-equivalence by Lemma \ref{leq}. By the given assumption, we see that $H_k(X)$ and $H_k(Y)$ are Hopfian groups. Hence $f_{XX}\in \AA^k(X)$ and $f_{YY}\in \AA^k(Y)$.

\item This case follows directly from the argument in part (ii).

\item For each $i\leq k$, the induced map $f_{*}\colon H_i(X)\oplus H_i(Y)\to H_i(X)\oplus H_i(Y)$ is an isomorphism. It follows from \cite[Theorem 3.2]{BCMADP} that $f_{XX}\in \AA^k(X)$ and $f_{YY}\in \AA^k(Y)$.
\end{enumerate}
\end{proof}

The following example illustrates instances of $k$-reducibility for self-maps in the monoid of several interesting CW-complexes.
\begin{example}\label{redu_example}
For each fix $k\geq 1$, the following results hold.:
\begin{enumerate}[(i)]
\item Let $m\neq n$, and consider the wedge sum $\E^m\vee \E^n$, where $\mathbb{E} = \mathbb{C}, \mathbb{H}$. Set $X = \E^m$ and $Y = \E^n$. The cohomology rings of $X$ and $Y$ are given by $$H^*(X) = \ZZ[\alpha]/(\alpha^{m+1}), ~~H^*(Y) = \ZZ[\beta]/(\beta^{n+1}),$$ where $\deg(\alpha) = \deg(\beta) = 2 \text{ or } 4 \text{ whenever } \mathbb{E} = \mathbb{C} \text{ or } \mathbb{H}$.
If $n > m$, observe that for any map $h\colon Y\to X$, the induced homomorphism $H^i(h)\colon H^i(X)\to H^i(Y)$ is trivial for all $i$. Indeed, on cohomology the induced homomorphism satisfies $h^*(\alpha) = l\beta$ for some $l\in \ZZ$. Then $h^*(\alpha^{m+1}) = l^{m+1}\beta^{m+1}$. It follows that $l^{m+1}\beta^{m+1} = 0$. Since $n>m$, the class $\beta^{m+1}$ is nonzero in $H^*(Y)$, and hence $l=0$. Thus $h^* = 0$.

If instead $m >n$, an analogous argument applies. In this case, for any map $g\colon X\to Y$, the induced homomorphism $H^i(g)\colon H^i(Y)\to H^i(X)$ is trivial for all $i$. Therefore, each self-map in $\AA^k(X\vee Y)$ is $k$-reducible by Lemma \ref{red}(iii). 

\item The wedge sum two different projective spaces $\C^m\vee \H^n$, where $m,n\in \mathbb{N}$. Suppose $X = \C^m$ and $Y = \H^n$. For a map $f\colon Y\to X$, the induced homomorphism $H^i(f)\colon H^i(X)\to H^i(Y)$ is trivial for all $i$. This follows from degree considerations of the generator of the cohomology rings. Indeed, $f^*(\alpha) = 0$, and hence $f^*(\alpha^i) = 0$ for all $i$. Consequently, each self-map in $\AA^k(X\vee Y)$ is $k$-reducible using Lemma \ref{red}(iii).

\item Consider the wedge sum $\E^n \vee S^m$, where $m,n\in \mathbb{N}$. 
First, let $\mathbb{E} = \mathbb{C}$. By arguments analogous to those in parts (i) and (ii), we observe that, for any pair of maps $f\colon \C^n\to S^m$ and $g\colon S^m\to \C^n$, at least one of the induced homomorphism $H^i(f)\colon H^i(S^m)\to H^i(\C^n)$ or $H^i(g)\colon H^i(\C^n)\to H^i(S^m)$ is trivial for all $i$, whenever $n\neq 1$ or $m\neq 2$. Therefore, every self-map in $\AA^k(\C^n \vee S^m)$ is $k$-reducible, except in the case $m=2n=2$.

Next, let $\mathbb{E} = \mathbb{H}$. By an analogous argument, each self-map in $\AA^k(\H^n \vee S^m)$ is $k$-reducible, except in the case $m = 4n =4$, from Lemma \ref{red}(iii).
 
\item Consider the wedge sum $\mathbb{C}P^3\vee (S^2\times S^4)$. Although the spaces $\mathbb{C}P^3$ and $S^2\times S^4$ have isomorphic homology and cohomology groups, but their cohomology ring structures are different. In particular, the cohomology rings are given by $$H^*(S^2\times S^4)\cong \ZZ[\alpha,\beta]/(\alpha^2,\beta^2) \text{ and } H^*(\mathbb{C}P^3)\cong \ZZ[\gamma]/(\gamma^4),$$ where $\deg(\alpha)= \deg(\gamma) = 2,~\deg(\beta) =4$. For a map $f\colon \C^3\to S^2\times S^4$, by an argument analogous to that in part (i), one obtains $f^*(\alpha) = 0$ and $f^*(\alpha\cdot \beta) = 0$. Therefore, $H^i(f)\colon H^i(S^2\times S^4)\to H^i(\mathbb{C}P^3)$ is trivial for $i=1,2,3,5,6$. Moreover, given a map $g\colon S^2\times S^4\to \C^3$, the induced homomorphism $H^i(g)\colon H^i(\mathbb{C}P^3)\to H^i(S^2\times S^4)$ is trivial for $i=1,3,4,5,6$. Because $g^*(\gamma^2) = (g^*(\gamma))^2 = a^2\alpha^2= 0$, where $a\in \ZZ$. Similarly, $g^*(\gamma^3) = (g^*(\gamma))^3= 0$. Consequently, for $k\geq 1$, every the self-map in $\AA^k\big(\mathbb{C}P^3\vee (S^2\vee S^4)\big)$ is $k$-reducible by Lemma \ref{red}(iii).

\item Let $M(G,m)$ and $M(H,n)$ be two Moore spaces, where $G,H$ are abelian groups and $m,n\in \ZZ$. If $m\neq n$, then each self-map in $\AA^k\big(M(G,m)\vee M(H,n)\big)$ is $k$-reducible, by Lemma \ref{red}(i).
 
\item Suppose $G$ and $H$ are two finite abelian groups with no common direct factors. By Lemma \ref{red}(iv), every self-map in $\AA^k\big(M(G,n)\vee M(H,n)\big)$ is $k$-reducible. For instance, take $G= \ZZ_3 \oplus \ZZ_4$ and $H= \ZZ_5 \oplus \ZZ_6$. Indeed, each self-maps in $\AA^k\big(M(\ZZ_3,n)\vee M(\ZZ_4,n)\vee M(\ZZ_5,n)\vee M(\ZZ_6,n)\big)$ is $k$-reducible.

\item Let $X$ be a simply connected CW-complex such that $H_n(X)\cong G$, where $G$ has no subgroup that is isomorphic to $\ZZ$. Moreover, $Y = M(H,n)$, where $H$ has no subgroup that is isomorphic to a finite group. Thus, $\Hom(H_i(X), H_i(Y)) = 0$ for all $i$. Hence, every self-map in $\AA^k(X\vee Y)$ $k$-is reducible by Lemma \ref{red}(i).
\end{enumerate}
\end{example}

Recall that the prime fields consist of the field of rational numbers $\mathbb{Q}$ and the finite fields $\mathbb{F}_p$ for primes $p$. By the universal coefficient theorems, we obtain the following lemma. The proof proceeds along the same lines as that of \cite[Theorem 4.10]{PPRSE} and is therefore omitted.

\begin{lemma}\label{lem:hom_prime fields}
Let $X$ be a CW-complex of finite type. For a self-map $f\colon X\to X$, the induced homomorphism $f_*\colon H_i(X)\to H_i(X) ~~(\text{ or } f^*\colon H^i(X)\to H^i(X))$ is an isomorphism if $f_*\colon H_i(X;\mathbb{F})\to H_i(X;\mathbb{F}) ~~(\text{ or } f^*\colon H^i(X;\mathbb{F})\to H^i(X;\mathbb{F}))$ is an isomorphism for every prime fields $\mathbb{F}$.
\end{lemma}

Note that Lemma \ref{lem:hom_prime fields} remains valid even when $X$ is a non-simply connected CW-complex of finite type. Moreover, given a positive integer $k$, we denote a monoid $\AA^k(X;\mathbb{F})$ by $$\AA^k(X;\mathbb{F}):=\{f\in [X,X]:f_*\colon H_i(X;\mathbb{F})\xrightarrow{\cong} H_i(X;\mathbb{F}), \text{ for all } i\leq k\}.$$
The following result is obtained by applying the universal coefficient theorem for cohomology, together with Lemma \ref{leq} and Lemma \ref{lem:hom_prime fields}.

\begin{proposition}\label{prop:redu_finite field}
Let $X$ and $Y$ be CW-complexes of finite type. For all prime fields $\mathbb{F}$, if any self-map in $\AA^k(X\vee Y;\mathbb{F})~~(\text{ or } \AAA_k(X\vee Y;\mathbb{F}))$ is $k$-reducible, then each self-map in both the monoids $\AA^k(X\vee Y), \text{ and } \AAA_k(X\vee Y)$ is $k$-reducible.    
\end{proposition}

In the following theorem, we obtain a sufficient condition for 
$k$-reducibility.
\begin{theorem}\label{thm:redu_cohom}
Let $X$ and $Y$ be finite type CW-complexes whose cohomology rings $H^*(X;\mathbb{F})$ and $H^*(Y;\mathbb{F})$ are polynomial rings in one indeterminate for each prime field $\mathbb{F}$. Suppose that $X$ and $Y$ have different connectivity or different dimensions. Then for $k\geq 1$, every self-map in the monoid $\AA^k(X\vee Y)$ is $k$-reducible.
\end{theorem}

\begin{proof}
Suppose that the elements of the cohomology rings $H^*(X)$ and $H^*(Y)$ are defined over different finite groups that do not have a common factor group of the form $\ZZ/{p^r\ZZ}$, even for different values of $r$.  Therefore, the $k$-reducibility of self-maps follows by Lemma \ref{red}(iv).  

Next, assume that for some prime field $\mathbb{F}$, both the cohomology rings $H^*(X;\mathbb{F})$ and $H^*(Y;\mathbb{F})$ are over the same finite field, say $\mathbb{F}_p$. Suppose the cohomology ring $H^*(X;\mathbb{F})$ is a polynomial ring of $\alpha$ and $H^*(Y;\mathbb{F})$ is a polynomial ring of $\beta$. Thus, $p\alpha = p\beta = 0$. If $\conn(X)<\conn(Y)$, then $\deg(\alpha) < \deg(\beta)$. Thus, $f^*(\alpha) = 0$ for any map $f\colon Y\to X$. Moreover, let $\conn(X) = \conn(Y)$ but $\dim(X)< \dim(Y)$. Now $f^*(\alpha) = l\beta$ for some $l\in \ZZ$. There exists an integer, say $m$ such that $\alpha^m = 0$ while $\beta^m\neq 0$. Applying $f^*$, we obtain $f^*(\alpha^m) = l^m\beta^m$ and this implies $l^m\beta^m = 0$. It follows that $p/l^m$ and so $p/l$. Therefore, $f^*(\alpha) = l\beta = 0$ and the map $f\colon Y\to X$ induces trivial homomorphism on cohomology. A similar argument applies when both cohomology rings are defined over $\mathbb{Q}$. Consequently, for every prime field 
$\mathbb{F}$, each self-map in $\AAA_k(X\vee Y, \mathbb{F})$ is $k$-reducible for $k\geq 1$ by using the similar argument as in the proof of Lemma \ref{red}(iii). Hence, by Proposition \ref{prop:redu_finite field}, we get the desired result. 
\end{proof}

\begin{remark}\label{rmk:RP^n}
There exist CW-complexes that are not simply connected, but the conclusion of Theorem \ref{thm:redu_cohom} is valid partially. As an illustrative example, one may take two real projective spaces of different dimensions. We know that 
$$H^*(\mathbb{R} P^r) \cong 
\begin{cases}
    \ZZ[\alpha]/(2\alpha,\alpha^{\frac{r}{2}+1}),~ \deg(\alpha)=2, & \text{ if $r$ is even }\\
    \ZZ[\alpha,\beta]/(2\alpha, \alpha^{\frac{r+1}{2}}, \beta^2,\alpha\beta), ~\deg(\alpha)=2, \deg(\beta)=r, & \text{ if $r$ is odd}.
\end{cases}
$$
Further, for a prime field $\mathbb{F}$, we have 
$$H^*(\mathbb{R}P^r;\mathbb{F}) \cong
\begin{cases}
    \mathbb{F}_2[\alpha]/(\alpha^{r+1}), ~\deg(\alpha) = 1, &\text{ if } \mathbb{F} = \mathbb{F}_2,\\
    \Lambda_{\mathbb{F}}[\alpha], ~\deg(\alpha) = r, &\text{ if } \mathbb{F}\neq \mathbb{F}_2, r \text{ is odd},\\
    \mathbb{F} &\text{ if } \mathbb{F}\neq \mathbb{F}_2, r \text{ is even}.
\end{cases}
$$
Next, consider two real projective spaces of different dimensions, namely $\mathbb{R}P^n$ and $\mathbb{R}P^m$, where $m\neq n$. Let $\alpha$ and $\beta$ denote generators of the cohomology rings $H^*(\mathbb{R}P^n;\mathbb{F})$ and $H^*(\mathbb{R}P^m;\mathbb{F})$, respectively. Without loss of generality, assume that $m>n$. For a map $g\colon \mathbb{R}P^m\to \mathbb{R}P^n$, the induce homomorphism $g^*\colon H^*(\mathbb{R}P^n;\mathbb{F})\to H^*(\mathbb{R}P^m;\mathbb{F})$ satisfies $$g^*(\alpha) = 
\begin{cases}
   l\beta, \text{ for some } l\in \ZZ, &\text{ if } \mathbb{F} = \mathbb{F}_2,\\
   0 &\text{ otherwise}.
\end{cases}
$$
This implies $0 = g^*(\alpha^{n+1}) = l^{n+1}\beta^{n+1}$ when $\mathbb{F} = \mathbb{F}_2$. It follows that $2/l^{n+1}$ and so $2/l$. Thus, $g^*(\alpha) = 0$, and $H^i(g,\mathbb{F})\colon H^i(\mathbb{R}P^n;\mathbb{F})\to H^i(\mathbb{R}P^m;\mathbb{F}) = 0$ for all $i$. An argument analogous to the proof of Lemma \ref{red}(iii), we see that every self-map in $\AAA_k(\mathbb{R} P^n\vee \mathbb{R}P^m;\mathbb{F})$ is $k$-reducible for all $k\geq 1$. Let $f\in \AA^m(\mathbb{R}P^n\vee \mathbb{R}P^m)$. By the universal coefficient theorem for cohomology, we have $f\in \AAA_m(\mathbb{R}P^n\vee \mathbb{R}P^m;\mathbb{F})$ for all prime field $\mathbb{F}$. Take $X = \mathbb{R}P^n$ and $Y = \mathbb{R}P^m$. Thus the components $f_{XX}\in \AAA_m(\mathbb{R}P^n; \mathbb{F}), f_{YY}\in \AAA_m(\mathbb{R}P^m;\mathbb{F})$ and hence so $f_{XX}\in \AAA_m(\mathbb{R}P^n), f_{YY}\in \AAA_m(\mathbb{R}P^m)$ by Lemma \ref{lem:hom_prime fields}. It follows, by \cite[Theorem 12, p.248]{Spanier} that $f_{XX}\in \AA^m(\mathbb{R}P^n), f_{YY}\in \AA^m(\mathbb{R}P^m)$. Consequently, any self-map in $\AA^m(\mathbb{R}P^n\vee \mathbb{R}P^m)$ is $m$-reducible. In addition, for $k\geq 1$, any self-map in $\AA^{2k}(\R^n\vee \R^m)$ is $k$-reducible whenever $m\neq n$.

Nevertheless, if $m\neq n$, then every self-map in $\Aut(\R^n\vee \R^m)$ is reducible.
\end{remark}

Let $X$ and $Y$ be simply connected CW-complexes of finite type. The study of $k$-reducibility of self-maps in $\AA^k(X\times Y)$ (or $\AAA_k(X\times Y)$) reduces to the corresponding problem for self-maps in $\A^k(X\times Y)$, which has already been investigated in \cite{DSMSE, LSPS, PPRSE}. Likewise, the verification of $k$-reducibility for self-maps in $\A^k(X\vee Y)$ is equivalent to verifying the $k$-reducibility of self-maps in $\AA^k(X\vee Y)$ (or $\AAA_k(X\vee Y)$).

%\begin{example}
%Consider the classifying space of the orthogonal group $BO(n)$. For a prime field $\mathbb{F}$, the cohomology ring is given by 
%$$H^*\big(BO(1);\mathbb{F}\big)\cong
%\begin{cases}
%    \mathbb{F}[\alpha], \deg(\alpha)=1  &\text{ if } \mathbb{F} = \mathbb{F}_2,\\
%    \mathbb{F}, &\text{ if } \mathbb{F}\neq \mathbb{F}_2.
%\end{cases}
%$$
%\end{example}

\begin{example}
The cohomology ring of the classifying space of the unitary group $BU(n)$ is given by $H^*(BU(n))\cong \ZZ[\alpha_1,\ldots,\alpha_n]$, where $\deg(\alpha_i) = 2i$.  Let $\mathbb{E} = \mathbb{R}, \mathbb{C},\mathbb{H}$. Arguing as in Example \ref{redu_example}(i), one observes that any map $f\colon BU(n)\to \mathbb{E}P^m$ induces a trivial homomorphism $H^i(f)\colon H^i(\mathbb{E}P^m)\to H^i(BU(n))$ for all $i,m,n$. It follows from Lemma \ref{red}(iii) that, for every $k\geq 1$, any self-map in $\AAA_k(BU(n)\vee \mathbb{E}P^m)$ is $k$-reducible for all $m,n$. Consequently, for $k\geq 1$, each self-map in $\AA^k(BU(n)\vee \mathbb{E}P^m)$ is $k$-reducible for all $m,n$. This conclusion follows from the universal coefficient theorem for cohomology together with the Lemma \ref{leq}, whenever $\mathbb{F} = \mathbb{C}, \mathbb{H}$. Further, by \cite[Theorem 12, p.248]{Spanier}, any self-map in $\AA^k(BU(n)\vee \R^m)$ is $k$-reducible for $k\geq m$.
\end{example}

\begin{example}
Consider two CW-complexes $S^2\times S^4\times \cdots \times S^{2m}$ and $\C^n\times \C^{n'}$, where $m$ could be any positive integer and $n,n'\geq 2$. For instance, take $m(m+1) = 2(n+n')$. The cohomology rings are given by $H^*(\mathbb{C}P^n\times \mathbb{C}P^{n'})\cong \ZZ[\alpha_1,\alpha_2]/(\alpha^{n+1}_1, \alpha^{n'+1}_2)$, where $\deg(\alpha_1) = \deg(\alpha_2) = 2$ and $H^*(S^2\times S^4\times \cdots \times S^{2m})\cong \ZZ[\beta_1,\ldots,\beta_m]/(\beta^2_1,\ldots,\beta^2_m)$, where $\deg(\beta_i) = 2i$ for all $1\leq i\leq m$. Notice that $H^{2i}(f)\colon H^{2i}(\C^n\times \C^{n'})\to H^{2i}(S^2\times S^4\times \cdots \times S^{2m})$ is trivial for all $2\leq i\leq n+n'$.
Because, for a map $f\colon S^2\times S^4\times \cdots \times S^{2m}\to \C^n\times \C^{n'}$, we obtain $f^*(\alpha^r_1\alpha^l_2) =f^*(\alpha_1)^rf^*(\alpha_2)^l = 0$, where $2\leq r+l\leq n+n'$ by an argument similar to that in Example \ref{redu_example} (iv). In addition, for a map $g\colon \C^n\times \C^{n'}\to S^2\times S^4\times \cdots \times S^{2m}$, we have $g^*(\beta_1) = 0$. Thus, $H^2(g)\colon H^2(S^2\times S^4\times \cdots \times S^{2m})\to H^2(\C^n\times \C^{n'})) = 0$.
Consequently, for $k\geq 1$, every self-map in $\AA^k\big(\mathbb{C}P^n\vee (S^2\times S^4\times S^{2m})\big)$ is $k$-reducible by Lemma \ref{red}(iii).
\end{example}

The conclusion of Lemma \ref{red}(iii) partially extends to the class of non-simply connected CW-complexes, as illustrated by the following example.
\begin{example}\label{exam:redu_multi_generators}
The cohomology ring of the special orthogonal group $SO(5)$ is given by $H^*(SO(5))\cong \ZZ[\alpha_1, \alpha_2, \alpha_3]/(2\alpha_1, \alpha_1^4, \alpha_2^4, \alpha_3^2, \alpha_1\alpha_3, \alpha_1^3-\alpha_2^2)$, where $\deg(\alpha_1)=2, \deg(\alpha_2) = 3, \deg(\alpha_3) = 7$. For a prime field $\mathbb{F}$, we have 
$$
H^*\big(SO(5);\mathbb{F}\big)\cong
\begin{cases}
    \mathbb{F}_2[\alpha_1,\alpha_2]/(\alpha^8_1,\alpha^2_2), &\text{ where } \deg(\alpha_1) =1, \deg(\alpha_2) = 3 ~\text{ for } \mathbb{F} = \mathbb{F}_2,\\
    \Lambda_{\mathbb{F}}[\alpha_1,\alpha_2], &\text{ where } \deg(\alpha_1) = 3, \deg(\alpha_2) = 7 ~\text{ for } \mathbb{F} \neq \mathbb{F}_2.
\end{cases}
$$
For each $i$, at least one of the homomorphisms $H^i(f;\mathbb{F})\colon H^i(SO(5);\mathbb{F})\to H^i(\mathbb{R}P^n;\mathbb{F})$ or $H^i(g;\mathbb{F})\colon H^i(\mathbb{R}P^n;\mathbb{F})\to H^i(SO(5);\mathbb{F})$ induced by any maps $f\colon \R^n\to SO(5)$ and $g\colon SO(5)\to \R^n$, is trivial whenever $n\neq 3, 7$. Therefore, by applying an argument analogous to that used in the proof of Lemma \ref{red}(iii), we conclude that any self-map in $\AAA_k\big(SO(5)\vee \mathbb{R}P^n;\mathbb{F}\big)$ is $k$-reducible for all $k\geq 1$. Consequently, for $k\geq \max\{n,8\}$, every self-map in $\AA^k\big(SO(5)\vee \mathbb{R}P^n\big)$ is $k$-reducible by the universal coefficient theorem for cohomology together with \cite[Theorem 12, p.248]{Spanier}. 

Similarly, for $\mathbb{E} = \mathbb{C}\text{ and } \mathbb{H}$, each self-map in $\AA^k(SO(5)\vee \mathbb{E}P^n)$ is $k$-reducible for all $k\geq 8$, as $\mathbb{E}P^n$ is simply connected. However, every self-maps in $\Aut(SO(5)\vee \mathbb{E}P^n)$ is reducible for all $\mathbb{E} = \mathbb{R}, \mathbb{C}, \mathbb{H}$.
\end{example}

\begin{example}
The cohomology ring of the special unitary group $SU(m)$ is given by $H^*(SU(m))\cong \Lambda_{\ZZ}[\alpha_3,\ldots, \alpha_{2m-1}]$, where $\deg(\alpha_i) = i$. Similarly, for the symplectic group $Sp(m)$, we have  $H^*(Sp(m))\cong \Lambda_{\ZZ}[\beta_3,\ldots,\beta_{4m-1}]$, where $\deg(\beta_i) = i$. Let $\mathbb{E} = \mathbb{C}\text{ or } \mathbb{H}$. Under the stated degree conditions on the generators, we see that $\Hom(H^i(SU(m)),H^i(\mathbb{E}P^n))=0$ and $\Hom(H^i(Sp(m)),H^i(\mathbb{E}P^n)=0$ for all odd integers $i$. Consequently, for $k\geq 1$, any self-map in $\AA^k(SU(m)\vee \mathbb{E}P^n)$ and $\AA^k(Sp(m)\vee \mathbb{E}P^n)$ is $k$-reducible for all integers $m$ and $n$ according to Lemma \ref{red}(iii). 

The situation is different for $\mathbb{E} = \mathbb{R}$. For $k\geq 1$, any self-map in $\AA^k(SU(m)\vee \R^n)$ is $k$-reducible, provided that either $n$ is even or $n>m^2-1$, by Lemma \ref{red}(i).

Moreover, consider the classifying space of the special unitary group $BSU(m)$ and the symplectic group $BSp(m)$, respectively. The cohomology rings are given by $H^*(BSU(m))\cong \ZZ[\alpha_2,\ldots,\alpha_m]$, where $\deg(\alpha_i) = 2i$ and $H^*(BSp(n))\cong \ZZ[\beta_1,\ldots, \beta_n]$, where $\deg(\beta_i) = 4i$. For degree reasons, any self-map in the monoid $\AA^k(BG\vee G')$ is $k$-reducible, where $G= SU(m) \text{ or } Sp(m)$ and $G' = SU(n) \text{ or } Sp(n)$ for all $m,n$. This conclusion follows directly from Lemma \ref{red}(iii).
\end{example}

%------------------------------------------------------

\section{Nilpotency and Reducibility} \label{nilpotent_reducibility}
In this section, we recall a notion called nilpotency of an endomorphism map that helps to determine the reducibility of self-maps in $\AA^k(X\vee Y)$. Let $G$ be an abelian group written additively. An endomorphism $\phi\colon G\to G$ is called \emph{quasi-regular} if the map $\Id-\phi\colon G\to G$ is an isomorphism. In general, the endomorphism ring $\End(G)$ is a non-commutative ring with unity $\Id\colon G\to G$ and zero element $0\colon G\to G$, the constant map. 
An endomorphism $\phi\colon G\to G$ is said to be nilpotent if there exists some integer $n$ such that $\phi^n = \phi\circ \cdots \circ\phi = 0$. It is well-known that for a nilpotent endomorphism $\phi$, both the maps $\Id - \phi$ and $\Id + \phi$ are units (isomorphisms) in $\End(G)$. It follows that a nilpotent endomorphism $\phi$ is always quasi-regular.
Further, \emph{Jacobson radical} of the ring $\End(G)$ is defined as $J(\End(G)):= \{\phi\in \End(G): \Id + (\psi\circ \phi\circ \varphi) ~\text{ is unit for all } \psi,\varphi\in \End(G)\}$. An endomorphism $\phi\colon G\to G$ is radical if it is in the Jacobson radical $J(\End(G))$.

\begin{definition}[{\cite[Section 2]{PPRSE}}]
A self-map $f\colon X\to X$ is said to be \emph{$n$-quasi-regular}, if the induced endomorphism $f_*\colon H_i(X)\to H_i(X)$ is \emph{quasi-regular} for $i\leq n$. 
\end{definition}
Therefore, a self-map $f\colon X\to X$ is quasi-regular if $n=\infty$.

The following theorem is a generalization of \cite[Proposition 2.1]{PPRSE} and \cite[Theorem 4.2]{LSPS}.
\begin{theorem}\label{redu_radical}
Assume that each self-map of $Y$ that factors (up to homotopy) through $X$ induces radical endomorphism on homology group of $Y$ up to the degree $k$. Then a self-map $f\in \AA^k(X\vee Y)$ is $k$-reducible if $f_{XX}\in \Aut(X).$
\end{theorem}

\begin{proof}
Let $f\in \AA^k(X\vee Y)$ such that $f_{XX}\in \Aut(X)$. It is sufficient to show that $f_{YY}\in \AA^k(Y).$ For $i\leq k$, the matrix $M_i(f)$ is invertible. Thus, there exists an isomorphism $\phi_i\colon H_i(X)\oplus H_i(Y)\to H_i(X)\oplus H_i(Y)$ such that $$M_i(f)\cdot M(\phi_i) = M(\phi_i)\cdot M_i(f) = M(\Id_{H_i(X)\oplus H_i(Y)}) ~\text{ for } i\leq k,$$
where \[
M(\phi_i) = 
\begin{bmatrix}
\phi^i_{11} & \phi^i_{12} \\
\phi^i_{21} & \phi^i_{22} 
\end{bmatrix}.
\]
It follows that $(f_{XX})_*\phi^i_{12} + (f_{XY})_*\phi^i_{22} = 0$, and $(f_{YX})_*\phi^i_{12} + (f_{YY})_*\phi^i_{22}= \Id_{H_i(Y)}$. Thus $\phi^i_{12} = -(f_{XX})^{-1}_*(f_{XY})_*\phi^i_{22}$, since $(f_{XX})_*\in \Aut(H_i(X))$. This implies $$\Big((f_{YY})_* - (f_{YX})_*(f_{XX})^{-1}_*(f_{XY})_*\Big)\phi^i_{22} = \Id_{H_i(Y)}.$$ We proceed similarly for the equality $M(\phi_i)\cdot M_i(f) = \Id_{H_i(X)\oplus H_i(Y)}$ and obtain
$$\phi^i_{21}(f_{XX})_* + \phi^i_{22}(f_{YX})_* = 0, ~\text{ and } \phi^i_{21}(f_{XY})_* + \phi^i_{22}(f_{YY})_* = \Id_{H_i(Y)}.$$
Therefore, $\phi^i_{22}\Big((f_{YY})_* - (f_{YX})_*(f_{XX})^{-1}_*(f_{XY})_*)\Big) = \Id_{H_i(Y)}.$ Hence $\phi^i_{22}\in \Aut(H_i(Y))$. Let $\bar{f}_{XX}$ be a homotopy inverse of $f_{XX}\in \Aut(X)$. Thus, $(f_{XX})^{-1}_* = (\bar{f}_{XX})_*$ and it implies

\begin{align*}
(f_{YY})_* = ~& (\phi^i_{22})^{-1} + (f_{YX})_* (\bar{f}_{XX})_*(f_{XY})_*\\ = ~& (\phi^i_{22})^{-1} + (f_{YX})_* (\bar{f}_{XX})_*(f_{XY})_*\\
= ~& (\phi^i_{22})^{-1} + (f_{YX}\circ \bar{f}_{XX}\circ f_{XY})_*\\
= ~& (\phi^i_{22})^{-1}\Big(\Id_{H_i(Y)}) + \phi^i_{22}(f_{YX}\circ \bar{f}_{XX}\circ f_{XY})_*\Big).
\end{align*}
By the given assumption $(f_{YX}\circ \bar{f}_{XX}\circ f_{XY})_*\colon H_i(Y)\to H_i(Y)$ is radical endomorphism for $i\leq k$. This implies $\Id_{H_i(Y)} + \phi^i_{22}(f_{YX}\circ \bar{f}_{XX}\circ f_{XY})_*\in \Aut(H_i(Y))$ and so $(f_{YY})_*\in \Aut(H_i(Y))$ for $i\leq k$. Hence, $f_{YY}\in \AA^k(Y)$. This completes the proof. 
\end{proof}

\begin{definition}[{\cite[Section 4]{PPRSE}}]
Let $X$ and $Y$ be two simply connected CW-complexes, and $n$ be a fixed positive integer. We say that $X$ and $Y$ are \emph{homologically $n$-distant} if any self-map of $X$ that factors (up to homotopy) through $Y$ induces a nilpotent endomorphism of $H_i(X)$ for $i\leq n$. 
\end{definition}
This definition is symmetric in the sense that if $X$ and $Y$ are homologically $n$-distant, then any self-map of $Y$ that factors (up to homotopy) through $X$ induces nilpotent endomorphism on $H_i(Y)$ for $i\leq n$. If this condition holds for all $n$, that is, for $n = \infty$, then $X$ and $Y$ are said to be homologically distant. In a similar manner, one defines the notion of homotopically or cohomologically $n$-distant spaces (with integer coefficients).

The following proposition presents an alternative formulation of \ref{redu_radical}. It restates the main conclusion in an equivalent but structurally refined manner.
\begin{proposition}\label{prop_redu_end}
Assume that $X, Y$ are homologically $k$-distant and for each $i\leq k$, the ring $\End(H_i(Y))$ is commutative. Then a self-map $f\in \AA^k(X\vee Y)$ is $k$-reducible if $f_{XX}\in \Aut(X)$.
\end{proposition}

\begin{proof}
Let $\bar{f}_{XX}$ be a homotopy inverse of $f_{XX}$ in $\Aut(X)$. By the given assumption the self-map $f_{YX}\circ \bar{f}_{XX}\circ f_{XY}\colon Y\to Y$ induces nilpotent endomorphism of $H_i(Y)$ for $i\leq k$. Further, the endomorphism $\phi^i_{22}\colon H_i(Y)\to H_i(Y)$ in the proof of Theorem \ref{redu_radical} commutes with $(f_{YX}\circ \bar{f}_{XX}\circ f_{XY})_*\colon H_i(Y)\to H_i(Y)$ for $i\leq k$. It follows that the map $\phi^i_{22}\circ (f_{YX}\circ \bar{f}_{XX}\circ f_{XY})_*\colon H_i(Y)\to H_i(Y)$ is nilpotent endomorphism and so quasi-regular for $i\leq k$. Therefore, $\Id_{H_i(Y)} + \phi^i_{22}\circ (f_{YX}\circ \bar{f}_{XX}\circ f_{XY})_*\in \Aut(H_i(Y))$. From the proof of Theorem \ref{redu_radical}, we obtain $(f_{YY})_*\in \Aut(H_i(Y))$ for $i\leq k$. Hence, $f_{YY}\in \AA^k(Y)$ and we get the desired result.
\end{proof}

\begin{corollary}\label{cor:k-reducibility}
Let $X$ and $Y$ be homologically $k$-distant. Suppose $H_i(Y)\cong \displaystyle\bigoplus^{m}_{j=1} \ZZ/{{p_j^{r_j}}\ZZ}$, where $p_j$ are distinct primes, for $i\leq k$. Then a self-map $f\in \AA^k(X\vee Y)$ is $k$-reducible if $f_{XX}\in \Aut(X)$.
\end{corollary}

\begin{proof}
For $i\leq k$, we have $$\End(H_i(Y))\cong \End\Big(\displaystyle\bigoplus^{m}_{j=1} \ZZ/{{p_j^{r_j}}\ZZ}\Big) \cong \displaystyle\bigoplus^{m}_{j=1} \End(\ZZ/{{p_j^{r_j}}\ZZ}) \cong \displaystyle\bigoplus^{m}_{j=1} \ZZ/{{p_j^{r_j}}\ZZ}.$$
Therefore, using Proposition \ref{prop_redu_end}, we get the desired result.
\end{proof}

The following lemma provides a sufficient condition for nilpotency of an endomorphism over finite groups.
\begin{lemma}\label{prime_nilpotent}
Let $H=\displaystyle\bigoplus^{m}_{s=1} \ZZ/{p^{r_s}_s\ZZ}$, where all $p_s$ are distinct primes. If $G$ is a group such that it does not contain any $\ZZ/{p^{r_s}_s\ZZ}$ as a direct summand, then every endomorphism of $H$ that factors through $G$ is nilpotent.
\end{lemma}

\begin{proof}
Let $\phi\colon H \to H$ be an endomorphism that factors through the group $G$. Such an endomorphism can be identified with the matrix $M(\phi) = (\phi_{ij})_{m\times m}$, where $\phi_{ij}\colon \ZZ/{p^{r_j}_j\ZZ}\to \ZZ/{p^{r_i}_i\ZZ}$ is the homomorphsim for $1\leq i,j\leq m$. For $i\neq j$, it is obvious that $\phi_{ij} = 0$. 
Consequently, for every $n\geq 1$, the matrix of the $n$-fold iterate $\phi^n$ has the form $M(\phi^n) = (\phi^n_{ij})_{m\times m}$ with $\phi^n_{ij} = 0$ whenever $i\neq j$. Further, each diagonal endomorphism $\phi_{ii}\colon \ZZ/{p^{r_i}_i\ZZ}\to \ZZ/{p^{r_i}_i\ZZ}$ also factors through $G$. This factorization is illustrated by the following commutative diagram below:
$$
\xymatrix{
\ZZ/{p^{r_i}_i\ZZ} \ar@{^{(}->} [r]^{\iota_i} & \displaystyle\bigoplus^{m}_{s=1} \ZZ/{p^{r_s}_s\ZZ} \ar[rr]^{\phi} \ar[dr]  && \displaystyle\bigoplus^{m}_{s=1} \ZZ/{p^{r_s}_s\ZZ} \ar[r]^{p_i} & \ZZ/{p^{r_i}_i\ZZ}\\
&& G. \ar[ur] &&
}
$$
From \cite[Proposition 4.15]{PPRSE}, we see that the endomorphism $\phi_{ii}\colon \ZZ/{p^{r_i}_i\ZZ}\to \ZZ/{p^{r_i}_i\ZZ}$ is nilpotent for $i=1,\ldots,m$. This implies that $\phi$ is a nilpotent endomorphism.
\end{proof}

In the following proposition, we explicitly determine the $k$-reducibility.
\begin{proposition}\label{nilpotent_reducibility_homdis}
Suppose $X = K(G,n)$, where $n\geq 2$. Let $Y$ be a simply connected CW-complex such that $H_n(Y)\cong \displaystyle\bigoplus^{m}_{j=1}\ZZ/{p_j^{r_j}\ZZ}$, where $p_j$ are distinct primes and $r_j\geq 1.$ If none of the factors $\ZZ/{p_j^{r_j}\ZZ}$ is a direct summand of $G$, then every self-map in $\AA^n(X\vee Y)$ is $n$-reducible.
\end{proposition}

\begin{proof}
Let $f\in \AA^n(X\vee Y)$. From Lemma \ref{red}(i), we have $f_{XX}\in \AA^{n-1}(X)$ and $f_{YY}\in \AA^{n-1}(Y)$. It is sufficient to show that $(f_{XX})_*\in \Aut(H_n(X))$ and $(f_{YY})_*\in \Aut(H_n(Y))$. By Lemma \ref{prime_nilpotent}, we see that $X$ and $Y$ are homologically $n$-distant. Suppose $\phi\colon H_n(X)\oplus H_n(Y)\to H_n(X)\oplus H_n(Y)$ is an inverse of the endomorphism $f_*\colon H_n(X)\oplus H_n(Y)\to H_n(X)\oplus H_n(Y)$. Therefore, $M_n(f)\cdot M(\phi) = M(\Id_{H_n(X)\oplus H_n(Y)})$ implies that 
\begin{align*}
&(f_{XX})_*\phi_{11} + (f_{XY})_*\phi_{21} = \Id_{H_n(X)}\\
&(f_{YX})_*\phi_{12} + (f_{YY})_*\phi_{22} = \Id_{H_n(Y)}\\
&(f_{XX})_*\phi_{12} + (f_{XY})_*\phi_{22} = 0.
\end{align*} 
For the homomorphism $\phi_{12}\colon H_n(Y)\to H_n(X) = G$, then there exists a map $\psi\colon Y\to X = K(G,n)$ such that $\psi_* = \phi_{12}.$ Using the equality above, together with the fact that $X,Y$ are homologically $n$-distant, we obtain $(f_{YY})_*\phi_{22} = \Id_{H_n(Y)} - (f_{YX}\circ \psi)_*\in \Aut(H_n(Y))$. It follows that $(f_{YY})_*\in \Aut(H_n(Y))$ and $\phi_{22}\in \Aut(H_n(Y))$. By an argument analogous to that used in the proof of Theorem \ref{redu_radical}, we have 
$$(f_{XX})_*\Big(\phi_{11} - \phi_{12} \phi^{-1}_{22} \phi_{21}\Big) = \Id_{H_n(X)} ~\text {and } \Big(\phi_{11} - \phi_{12} \phi^{-1}_{22} \phi_{21}\Big)(f_{XX})_* = \Id_{H_n(X)}.$$ Consequently, $(f_{XX})_*\in \Aut(H_n(X))$. Hence, $f_{XX}\in \AA^n(X)$ and $f_{YY}\in\AA^n(Y)$. This completes the proof.
\end{proof}

\begin{example}\label{exam:redu_Eilenberg_Real Projective}
For $n\geq 1$, consider the suspension $\Sigma \mathbb{R}P^{2n}$ of real projective space and the Eilenberg–Mac Lane space $K(G,2n)$, where $G$ is an abelian group except $\ZZ/{2\ZZ}$. We know that $H_{2n}(\Sigma \mathbb{R}P^{2n})\cong H_{2n-1}(\mathbb{R}P^{2n})\cong \ZZ/{2\ZZ}$. Therefore, any self-map in $\AA^{2n}(K(G,2n)\vee \Sigma \mathbb{R}P^{2n})$ is $2n$-reducible by Proposition \ref{nilpotent_reducibility_homdis}. In fact, a stronger statement holds. For every $k\geq 1$, each self-map in $\AA^k(K(G,2n)\vee \Sigma \mathbb{R}P^{2n})$ is $k$-reducible. This follows from the fact that $H_i(\Sigma \mathbb{R}P^{2n}) \cong 0$ for all $i\geq 2n+1$ together with Lemma \ref{red}(i). 

Indeed, if either $n$ is odd or $G$ is an abelian group other than $\ZZ/{2\ZZ}$, then every self-map in $\AA^n(K(G,n)\vee \Sigma \mathbb{R}P^{2m})$ is $n$-reducible.
\end{example}

In the following lemma, we establish relations among the homotopical, homological, and cohomological notions of $n$-distance. This result extends the earlier statements appearing in \cite[Corllary 4.11, Corollary 4.12]{PPRSE}.
\begin{lemma}\label{lem:nilpotency}
Let $X$ and $Y$ be CW-complexes of finite types. Then the  following statements hold:
\begin{enumerate}[(i)]
\item $X$ and $Y$ are homologically $n$-distant if and only if they are homotopically $n$-distant.

\item If $X$ and $Y$ are homologically and homotopically $n$-distant with coefficients in every prime field $\mathbb{F}$, then they are homologically $n$-distant.

\item If $X$ and $Y$ are homologically $n$-distant, then they are cohomologically $n$-distant. In particular, they are cohomologically $n$-distant with coefficients in any prime field $\mathbb{F}$.

\item If $X$ and $Y$ are cohomologically $n$-distant or cohomologically $n$-distant with coefficients in every prime field $\mathbb{F}$, then they are homologically $(n-1)$-distant.

\item If $X$ and $Y$ are homotopically $n$-distant, then they are homotopically $(n-1)$-distant with coefficients in any prime field $\mathbb{F}$.
\end{enumerate}
\end{lemma}

\begin{proof}
\begin{enumerate}[(i)]
\item This follows directly from \cite[Theorem 3.5]{PPRSE}.

\item For a prime field $\mathbb{F}$, the spaces $X$ and $Y$ are homologically $n$-distant with coefficients in $\mathbb{F}$ if and only if they are homotopically $n$-distant with coefficients in 
$\mathbb{F}$, by the same argument as in part (i). The remaining assertion follows from \cite[Theorem 4.10]{PPRSE}.

\item By the universal coefficient theorems for cohomology, together with \cite [Lemma 3.1]{PPRSE}, the conclusion follows. 
 
\item If $X$ and $Y$ are cohomologically $n$-distant with coefficient in all prime fields $\mathbb{F}$, then they are cohomologically $n$-distant by an argument similar to that used in the proof of \cite[Theorem 4.10]{PPRSE}. Moreover, applying \cite[Theorem 12, p.248]{Spanier} together with \cite[Lemma 3.1]{PPRSE} yields the desired result.

\item Finally, \cite[Proposition 4H.2]{AHAT}, together with \cite[Lemma 3.1]{PPRSE} implies that $X$ and $Y$ are homotopically $(n-1)$-distant with coefficient in any prime field $\mathbb{F}$. This completes the proof.
\end{enumerate}
\end{proof}

\begin{example}
As an application of Lemma \ref{lem:nilpotency}, we get the following observations: 
\begin{enumerate}[(i)]
\item The CW-complexes $\mathbb{C}P^3$ and $S^2\times S^4$ discussed in Example \ref{redu_example}(iv), are cohomologically $n$-distant for all $n\geq 1$. Consequently, they are also homotopically and homologically $n$-distant for all $n\geq 1$.

\item Consider the Eilenberg MacLance spaces $K(\ZZ/{p^r\ZZ},m)$ and $K(\ZZ/{q^s\ZZ},m')$, where $p,q$ are primes, not necessarily distinct. If at least one of the pairs $(m,m'), (p,q)$, $(r,s)$ is distinct, then these Eilenberg–MacLane spaces are homotopically $n$-distant for all $n\geq 1$.  Thus, they are also homologically and cohomologically $n$-distant for all $n\geq 1$.
\end{enumerate}
\end{example}

For an element $\alpha$ in a cohomology ring $H^*(X;R)$, we define $\cp(\alpha)$ to be the maximal integer $m$ such that $\alpha^m\neq 0$, where $R$ is a commutative ring. 

The following theorem generalizes Lemma \ref{red} and Theorem \ref{thm:redu_cohom}.
\begin{theorem}\label{thm_redu_multi_generators}
Let $X$ and $Y$ be finite type CW-complexes. Suppose that for every prime field $\mathbb{F}$, the cohomology rings $H^*(X;\mathbb{F})$ and $H^*(Y;\mathbb{F})$ are polynomial rings generated by finitely many indeterminates, say $\alpha_1,\ldots,\alpha_n$ and $\beta_1,\ldots,\beta_m$, respectively. Assume that for every pair of indices $i$ and $j$, either $\deg(\alpha_i) \neq \deg(\beta_j)$ or $\cp(\alpha_i)\neq \cp(\beta_j)$. Then any self-map in $\Aut(X\vee Y)$ is reducible. Moreover, the spaces $X$ and $Y$ are cohomologically distant. (Here, the degree sequence $\{\deg(\alpha_i)\}^n_{i=1}$  and $\{\deg(\beta_j)\}^m_{j=1}$ are assumed to be in non-decreasing order).
\end{theorem}

\begin{proof}
Suppose that the first coincidence in the degrees of the generators occurs at $i = i_1,\ldots, i_r$ and $j = j_1,\ldots, j_l$, i.e.,
\begin{equation}\label{eq:generator_equality}
\deg(\alpha_i) = \deg(\beta_j) = u, \text{ for } 1\leq i_1\leq i\leq i_r\leq n, 1\leq j_1\leq j\leq j_l\leq m.  
\end{equation}
Assume that there exist positive integers $m_i$  and $m'_j$ such that $\cp(\alpha_i) = m_i$ and $\cp(\beta_j) = {m'_j}$, with $m_i\neq m'_j$. We remark that it may also happen that certain products of the generators $\beta_j$ or the products of the generators $\alpha_i$ have degrees $u$; this situation is treated separately in Case IV. For clarity of exposition, we proceed in a linear manner in the first three cases.

\noindent \textbf{Case-I}:  Let $m_i< m'_j$ for all $i_1\leq i\leq i_r$ and $j_1\leq j\leq j_l$. Note that some of the values $m'_j$ may also be infinite. For a homomorphism $\phi\colon H^*(X;\mathbb{F})\to H^*(Y;\mathbb{F})$, we have $\phi(\alpha_i) = b_{ij_1}\beta_{j_1} + \cdots + b_{ij_l}\beta_{j_l}$, where $b_{ij}\in \ZZ$ for $i_1\leq i\leq i_r$ and $j_1\leq j\leq j_l$. Using the ring homomorphism, we obtain $$0 = \phi(\alpha^{m_i+1}_i) = (b_{ij_1}\beta_{j_1} + \cdots + b_{ij_l}\beta_{j_l})^{m_i+1} = \sum_{n_{j_1}+\cdots+n_{j_l}=m_i+1}\frac{(m_i+1)!}{n_{j_1}!\cdots n_{j_l}!}\prod^{j_l}_{j=j_1}(b_{ij}\beta_j)^{n_j}.$$ This implies that $b^{m_i+1}_{ij}\beta^{m_i+1}_j = 0$ for $j=j_1,\ldots,j_l$. It follows that $b^{m_i+1}_{ij}\equiv 0 \pmod{p}$ and so $b_{ij} \equiv 0 \pmod{p}$ for $\mathbb{F} = \mathbb{F}_p$. In other case $b_{ij} = 0$, whenever $\mathbb{F} = \mathbb{Q}$. Therefore, $\phi(\alpha_i) = 0$ for $i_1\leq i\leq i_r$. 

\vspace{0.5 em}

\noindent \textbf{Case-II}: Suppose that $m_i > m'_j$ for all $i_1\leq i\leq i_r$ and $j_1\leq j\leq j_l$. Note that some of the values $m_i$ may also be infinite. By an argument analogous to that given in Case-I, it follows that a ring homomorphism $\psi\colon H^*(Y;\mathbb{F})\to H^*(X;\mathbb{F})$ gives $\psi(\beta_j) = 0$ for $j_1\leq j\leq j_l$. 

Therefore, given an isomorphism $f^*\colon H^u(X;\mathbb{F})\oplus H^u(Y;\mathbb{F})\to H^u(X;\mathbb{F})\oplus H^u(Y;\mathbb{F})$ induced by a self-map $f\colon X\vee Y\to X\vee Y$, we conclude that $(f_{XX})^*\in \Aut(H^u(X;\mathbb{F}))$ and $(f_{YY})^*\in \Aut(H^u(Y;\mathbb{F}))$ combining the above two cases, together with Lemma \ref{red}(iii).

Indeed, if all the generators of the cohomology rings and their corresponding cup length relations are precisely those described in the preceding two cases, then every self-map in $\AAA_k(X\vee Y;\mathbb{F})$ is $k$-reducible for $k\geq 1$.

\vspace{0.5 em}

\noindent \textbf{Case-III}: Assume that the collections of integers $\{m_i\}^{i_r}_{i=i_1}$ and $\{m'_j\}^{j_l}_{j=j_1}$ are arranged in non-decreasing order; otherwise, we may reorder them without loss of generality. Suppose further that the integers $m_i$ and $m'_j$ satisfy a mixed system of inequalities. Without loss of generality, we may assume that $$m_{i_1}< m'_{j_1}<m_{i_2}<m'_{j_2}<m_{i_3}\leq \cdots \leq m_{i_r}< m'_{i_3}\leq \cdots \leq m'_{j_l}.$$ For a homomorphism $\phi\colon H^*(X;\mathbb{F})\to H^*(Y;\mathbb{F})$, proceeding with the similar argument as in Case-I, we obtain 
$$\phi(\alpha_i) = 
\begin{cases}
    0 & \text{ for } i = i_1,\\
    b_{i1}\beta_{j_1} & \text{ for } i = i_2,\\
    b_{i1}\beta_{j_1} + b_{i2}\beta_{j_2} & \text{ for } i_3\leq i\leq i_r.
\end{cases}
$$
Further, for a homomorphism $\psi\colon H^*(Y;\mathbb{F})\to H^*(X;\mathbb{F})$, we get 
$$\psi(\beta_j) = 
\begin{cases}
    a_{1j}\alpha_{i_1}, & \text{ for } j = j_1,\\
    a_{1j}\alpha_{i_1} + a_{2j}\alpha_{i_2} & \text{ for } j = j_2,\\
    a_{1j}\alpha_{i_1} + \cdots + a_{rj}\alpha_{i_r} & \text{ for } j_3\leq j\leq j_l.
\end{cases}
$$ 
This implies  that $\phi\circ \psi\colon H^u(Y;\mathbb{F})\to H^u(Y;\mathbb{F})$ and $\psi\circ \phi\colon H^u(X;\mathbb{F})\to H^u(X;\mathbb{F})$ are nilpotent endomorphisms. Let $f^*\colon H^u(X;\mathbb{F})\oplus H^u(Y;\mathbb{F})\to H^u(X;\mathbb{F})\oplus H^u(Y;\mathbb{F})$ be an isomorphism induced by a self-map $f$ in $\Aut(X\vee Y)$. Then there exists an isomorphism $\bar{f}^*\colon H^u(X;\mathbb{F})\oplus H^u(Y;\mathbb{F})\to H^u(X;\mathbb{F})\oplus H^u(Y;\mathbb{F})$ such that $M^u(f)\cdot M^u(\bar{f}^*) = M(\Id_{H^u(X;\mathbb{F})\oplus H^u(Y;\mathbb{F})})$, where $\bar{f}$ is the homotopy inverse of $f$. It follows that $(f_{XX})^* (\bar{f}_{XX})^* + (f_{YX})^*(\bar{f}_{XY})^* = \Id_{H^u(X;\mathbb{F})}$ and $(f_{XY})^*(\bar{f}_{YX})^* + (f_{YY})^*(\bar{f}_{YY})^* = \Id_{H^u(Y;\mathbb{F})}$. Since the endomorphisms $(f_{YX})^*(\bar{f}_{XY})^*$ and $(f_{XY})^*(\bar{f}_{YX})^*$ are nilpotent, this yields $$(f_{XX})^* \Phi^{11}_u  = \Id_{H^u(X;\mathbb{F})} - (f_{YX})^*\Phi^{12}_u\in \Aut(H^u(X;\mathbb{F}))$$ and $$(f_{YY})^*\Phi^{22}_u = \Id_{H^u(Y;\mathbb{F})} - (f_{XY})^*\Phi^{21}_u\in \Aut(H^u(Y;\mathbb{F})).$$  Hence, $(f_{XX})^*\in \Aut(H^u(X;\mathbb{F}))$ and $(f_{YY})^*\in \Aut(H^u(Y;\mathbb{F}))$, since $H^u(X;\mathbb{F})$ and $H^u(Y;\mathbb{F})$ are finitely generated.

\vspace{0.5 em}

\noindent \textbf{Case-IV}: If possible, let 
\begin{equation}\label{eq:generator_product_equality}
\deg(\alpha_i) = \deg(\beta_j) = \deg(\alpha_{i_0}\alpha_{i'_0}) = \deg(\beta_{j_0}\beta_{j'_0}) = u,
\end{equation}
where $1\leq i_0,i'_0<i_1,~1\leq j_0,j'_0<j_1$ and $i_1\leq i\leq i_r,~ j_1\leq j\leq j_l$. Thus, $\deg(\alpha_{i_0})<u,~\deg(\alpha_{i'_0})<u$ and $\deg(\beta_{j_0})<u,~\deg(\beta_{j'_0})<u$. Let $\phi\colon H^*(X;\mathbb{F})\to H^*(Y;\mathbb{F})$ and $\psi\colon H^*(Y;\mathbb{F})\to H^*(X;\mathbb{F})$ be two homomorphisms. By the minimality assumption on the equal degree in \eqref{eq:generator_equality}, we obtain $\phi(\alpha_{i_0}) = \phi(\alpha_{i'_0}) = 0$ and $\psi(\beta_{j_0}) = \psi(\beta_{j'_0}) = 0$. Observe that elements $\beta_{j_0}\beta_{j'_0}$ and $\alpha_{i_0}\alpha_{i'_0}$ lie in the images of some homomorphisms $\phi(\alpha_i)$ and  $\psi(\beta_j)$, respectively. Consequently, together with the three cases discussed above and the equalities $\phi(\alpha_{i_0}\alpha_{i'_0}) = 0$ and $\psi(\beta_{j_0}\beta_{j'_0}) = 0$, it follows that the induced endomorphisms $\phi\circ \psi\colon H^u(Y;\mathbb{F})\to H^u(Y;\mathbb{F})$ and $\psi\circ \phi\colon H^u(X;\mathbb{F})\to H^u(X;\mathbb{F})$ are nilpotent. 

Further, suppose that a relation analogous to \eqref{eq:generator_product_equality} occurs to a higher degree, say  $v$. In this situation, the elements, say $\phi(\alpha_{i_0}),~ \phi(\alpha_{i'_0}),~\psi(\beta_{j_0})$ and $\psi(\beta_{j'_0})$ can be expressed as sums of elements of strictly lower degree than those of $\alpha_{i_0}\alpha_{i'_0}$ and $\beta_{j_0}\beta_{j'_0}$, which need not vanish identically. By induction on the degree, there exist positive integers $s$ and $s'$ such that $(\psi\circ \phi)^s(\alpha_{i_0}\alpha_{i'_0}) = (\psi\circ \phi)^s(\alpha_{i_0}) (\psi\circ \phi)^s(\alpha_{i'_0})= 0$ and $(\phi\circ \psi)^{s'}(\beta_{j_0}\beta_{j'_0}) = (\phi\circ \psi)^{s'}(\beta_{j_0}) (\phi\circ \psi)^{s'}\beta_{j'_0}) = 0$. Consequently, $\phi\circ \psi\colon H^v(Y;\mathbb{F})\to H^v(Y;\mathbb{F})$ and $\psi\circ \phi\colon H^v(X;\mathbb{F})\to H^v(X;\mathbb{F})$ are nilpotent. Note that if, instead of two factors, relation \eqref{eq:generator_product_equality} involves products of more than two generators, the argument proceeds in an entirely analogous manner. Moreover, arguing as in Case III, for any self-map $f\in \Aut(X\vee Y)$, the induced homomorphisms $(f_{XX})^*\in \Aut(H^v(X;\mathbb{F})$
and $(f_{YY})^*\in \Aut(H^v(Y;\mathbb{F})$.

Combining the analysis of all four cases with Lemma \ref{lem:hom_prime fields}, we conclude that every self-map in $\Aut(X\vee Y)$ is reducible.

In addition, the spaces $X$ and $Y$ are cohomologically $n$-distance for all $n\geq 1$. Hence, they are also homotopically and homologically distant by Lemma \ref{lem:nilpotency}.
\end{proof}

%Further, fixed an index $i_0\in \{1,\ldots,n\}$. Suppose $\deg(\alpha_{i_0}) \neq \deg(\beta_j)$ for all $j$. Thus, $\Hom\big(H^{i_0}(X;\mathbb{F}),H^{i_0}(Y;\mathbb{F})\big) = 0$ and this implies that any isomorphism $f^*\colon H^{i_0}(X;\mathbb{F})\oplus H^{i_0}(Y;\mathbb{F})\to H^{i_0}(X;\mathbb{F})\oplus H^{i_0}(Y;\mathbb{F})$ implies $(f_{XX})^*\in \Aut(H^{i_0}(X;\mathbb{F}))$ and $(f_{YY})^*\in \Aut(H^{i_0}(Y;\mathbb{F}))$ by Lemma \ref{red}(ii).

\begin{example}
Consider the special orthogonal group $SO(n)$, where $n\geq 2$. For a prime field $\mathbb{F}$, the cohomology ring is given by 
$$H^*\big(SO(n);\mathbb{F}\big)\cong
\begin{cases}
   \mathbb{F}_2[\alpha_1,\ldots,\alpha_{\lfloor\frac{n}{2}\rfloor}]/(\alpha^{r_1}_1,\ldots,\alpha^{r_{\lfloor\frac{n}{2}\rfloor}}_{\lfloor\frac{n}{2}\rfloor}), & \deg(\alpha_i) = 2i-1, \text{ if } \mathbb{F}=\mathbb{F}_2,\\
   \Lambda_{\mathbb{F}}[\alpha_1,\ldots,\alpha_{\lfloor\frac{n}{2}\rfloor}], & \deg(\alpha_i) = 4i-1, \text{ if } \mathbb{F}\neq \mathbb{F}_2,\\ & \qquad \text{ and } n \text{ is odd},\\
   \Lambda_{\mathbb{F}}[\alpha_1,\ldots,\alpha_{\lfloor\frac{n}{2}\rfloor-1},\alpha], & \deg(\alpha_i) = 4i-1, \text{ if } \mathbb{F}\neq \mathbb{F}_2,\\ & \qquad \text{ and } n \text{ is even},
\end{cases}
$$
where $\deg(\alpha) = n-1$ and $r_i$ is the smallest power of $2$ such that $(2i-1)r_i\geq n$. Here, $\lfloor\frac{n}{2}\rfloor$ is the greatest integer less than or equal to $\frac{n}{2}$. 

Further, consider the real stiefel manifold $V_t(\mathbb{R}^m)$, where $1\leq t<m$ and $m\geq 2$. The cohomology ring is given by 
$$H^*(V_t(\mathbb{R}^m);\mathbb{F})\cong
\begin{cases}
    \mathbb{F}_2[\beta_1,\ldots, \beta_t]/\mathcal{I}_{m,t}, & \deg(\beta_i) = m-t+i-1,\\ & \qquad \text{ if } \mathbb{F}=\mathbb{F}_2,\\
    \Lambda_{\mathbb{F}}[\beta_1,\ldots, \beta_{\lfloor \frac{t+1}{2}\rfloor}], & 
    \deg(\beta_i) = 2(m-t)+ 4i-3, \\ & \qquad \text {if } \mathbb{F}\neq \mathbb{F}_2, m \text{ odd}, t \text{ even},\\ 
    \Lambda_{\mathbb{F}}[\beta_1,\ldots, \beta_{\lfloor \frac{t+1}{2}\rfloor-1},\beta], & \deg(\beta_i) = 2(m-t)+ 4i-3,\\ & \qquad \text {if } \mathbb{F}\neq \mathbb{F}_2, m \text{ even}, t \text{ odd},\\
    \Lambda_{\mathbb{F}}[\beta',\beta_2,\ldots, \beta_{\lfloor \frac{t+1}{2}\rfloor}], & \deg(\beta_i) = 2(m-t)+4i-5,\\ & \qquad \text{if } \mathbb{F}\neq \mathbb{F}_2, m \text{ odd}, t \text{ odd},\\
    \Lambda_{\mathbb{F}}[\beta',\beta_2,\ldots, \beta_{\lfloor \frac{t+1}{2}\rfloor-1},\beta], & \deg(\beta_i) = 2(m-t)+4i-5,\\ & \qquad \text{if } \mathbb{F}\neq \mathbb{F}_2, m \text{ even}, t \text{ even},
\end{cases}
$$
where $\deg(\beta') = m-t$ and $\deg(\beta) = m-1$. Moreover, $\mathcal{I}_{m,t}$ is an ideal generated by the following relations:
$$\beta^2_i = 
\begin{cases}
\beta_{2i+m-t-1}~ &\text{ if } i\leq \frac{2t-m+1}{2},\\
0~ &\text{ if } i\geq \frac{2t-m+2}{2}.
\end{cases}
$$
Notice that $H^*(V_t(\mathbb{R}^m);\mathbb{F}_2)\cong \mathbb{F}_2(\beta_1,\ldots,\beta_t)/(\beta^2_1,\ldots,\beta^2_t)$ whenever $m\geq 2t$. 

By Theorem \ref{thm_redu_multi_generators}, every self-map in $\Aut\big(SO(n)\vee V_t(\mathbb{R}^m)\big)$ is reducible provided that $(m-t)\geq n\geq 2$. In addition, the manifolds $SO(n)$ and $V_t(\mathbb{R}^m)$ are cohomologically distant. It follows from Lemma \ref{lem:nilpotency} that they are homologically and homotopically distant. More precisely, for $k\geq 1$, any self-map in $\AAA_k\big(SO(n)\vee V_t(\mathbb{R}^m);\mathbb{F}\big)$ is $k$-reducible for all prime field $\mathbb{F}$, by Lemma \ref{red}(iii). This is a consequence of the fact that, whenever $(m-t)\geq n\geq 2$, the degrees of the cohomology generators of $SO(n)$ are strictly smaller than those of $V_t(\mathbb{R}^m)$. Consequently, for $k\geq \frac{n(n-1)}{2}$, every self-map in $\AA^k\big(SO(n)\vee V_t(\mathbb{R}^m)\big)$ is $k$-reducible by \cite[Theorem 12, p.248]{Spanier}.

For $n\geq 2$, taking the classifying space of the special orthogonal group $BSO(n)$. The cohomology ring is given by 
$$H^*\big(BSO(n);\mathbb{F}\big)\cong
\begin{cases}
    \mathbb{F}[\gamma_2, \gamma_3, \ldots, \gamma_n], &\deg(\gamma_i)=i,  \text{ if } \mathbb{F} = \mathbb{F}_2,\\
    \mathbb{F}[\gamma_1,\ldots, \gamma_{\lfloor \frac{n}{2}\rfloor}], &\deg(\gamma_i) = 4i, \text{ if } \mathbb{F}\neq \mathbb{F}_2, n \text{ odd},\\
    \mathbb{F}[\gamma_1,\ldots,\gamma_{\lfloor \frac{n}{2}\rfloor-1},\gamma], &\deg(\gamma_i) = 4i,~\deg(\gamma) =n, \text{ if } \mathbb{F} \neq \mathbb{F}_2, n \text{ even}.
\end{cases}
$$
Note that if $\mathbb{F}\neq \mathbb{F}_2$ and $n$ is an even integer, then the top Pontryagin class $\gamma_{\lfloor \frac{n}{2}\rfloor}$ is related to the Euler class $\gamma$ by the identity $\gamma^2 = \gamma_{\lfloor \frac{n}{2}\rfloor}$. For more details, see \cite{AHAT,MTTL}. 
\end{example}

By Theorem \ref{thm_redu_multi_generators}, any self-map in $\Aut\big(SO(n)\vee BSO(m)\big)$ is reducible for all $n,m$. Indeed, for $k\geq 1$, each self-map in $\AAA_k\big(SO(n)\vee BSO(m);\mathbb{F}\big)$ is $k$-reducible for all $n,m$ and for every prime field $\mathbb{F}$, in view of Case-I of Theorem \ref{thm_redu_multi_generators}. Consequently, if $k\geq \frac{n(n-1)}{2}$, then every self-map in $\AA^k\big(SO(n)\vee BSO(m)\big)$ is $k$-reducible for all $n,m$ by universal coefficient theorem for cohomology together with \cite[Theorem 12, p.248]{Spanier}.

Further, by combining the observation of Case-II of Theorem \ref{thm_redu_multi_generators} with Proposition \ref{prop:redu_finite field}, we deduce that for $k\geq 1$, any self-map in $\AA^k\big(BSO(n)\vee V_t(\mathbb{R}^m)\big)$ is $k$-reducible for all integers $n,m$ and $t$.

In particular, according to the hypotheses of Case-III in Theorem \ref{thm_redu_multi_generators}, any self-map in $\Aut\big((\C^2\times \C^4)\vee V_4(\mathbb{R}^6)\big)$ is reducible. However, for $k\geq 1$, every self-map in $\AA^k\big((\C^3\times \C^9\times \H^2)\vee V_6(\mathbb{R}^9)\big)$ is $k$-reducible by Lemma \ref{red}(iii).

As an additional remark, let $X$ be a simply connected CW-complex of finite type whose cohomology ring possesses no generators of cup-length equal to two. Suppose that $m\geq 2t$. Then, by Case-II of Theorem \ref{thm_redu_multi_generators}, every self-map in $\AA^k\big(X\vee V_t(\mathbb{R}^m)\big)$ is $k$-reducible for all $k\geq 1$.

%----------------------------------------------------------
\sect{Homology self-closeness number}\label{self-closeness number}
For a simply connected, finite-dimensional CW-complex $X$, it is well known that $$\conn(X)+1 \leq\NAA(X)\leq H_{*}\text{-}\dim(X)\leq \dim(X),$$ where $H_{*}\text{-}\dim(X):= \max\{i\geq 0: H_i(X)\neq 0\}$. In this section, we investigate the homology self-closeness number of wedge sums of spaces. In particular, we establish several relationships among the homology self-closeness numbers of wedge sums, products, and smash products of spaces. For a simply connected CW-complex $X$ of finite type, it follows from \cite[Theorem 41]{OYSCFC} that $\NA(X) = \NAA(X)\leq \NAAA(X)$. Thus, in the simply connected setting, the notions of (homotopy) self-closeness and homology self-closeness numbers coincide, provided that the CW-complex is of finite type. In contrast, when 
$X$ is not simply connected, the invariants
$\NA(X)$ and $\NAA(X)$ may differ. For instance, take $\mathbb{R}P^{2n}$. Since $H_{*}\text{-}\dim(\mathbb{R}P^{2n}) = 2n-1$, then $\NAA(\mathbb{R}P^{2n})\leq 2n-1$. Moreover, by \cite[Theorem 13]{OYSCFC}, we have $\NA(\mathbb{R}P^{2n}) =2n$.  

In the following theorem, we investigate the homology self-closeness numbers of wedge sums and cartesian products of CW-complexes in terms of the corresponding invariants of the individual spaces. We emphasize that the method of proof adopted here is different from that used in \cite[Theorem 3.3]{LSPS}.
\begin{theorem}\label{equality_self-closeness_wedge_product}
Let $X$ and $Y$ be CW-complexes. Then 
\begin{enumerate}[(a)]
\item $\NAA(X\vee Y)\geq \max\{\NAA(X), \NAA(Y)\}$.
\item $\NAA(X\times Y)\geq \max\{\NAA(X), \NAA(Y)\}$.
\end{enumerate}
Moreover, assume that $m = \max\{\NAA(X), \NAA(Y)\}$. If each map in $\AA^m(X\vee Y)$ or $\AA^m(X\times Y)$ is $m$-reducible then we get the equality in (a) or (b). 
\end{theorem}

\begin{proof}
\begin{enumerate}[(a)]
\item If $\NAA(X\vee Y)$ is infinite then we get the result trivially. Otherwise, assume that $\NAA(X\vee Y) = n < \infty$. Let $g\in \AA^n(X)$ and $h\in \AA^n(Y)$. It follows that $g\vee h\in \AA^n(X\vee Y) = \Aut(X\vee Y)$. Thus $(g\vee h)_* \colon H_i(X\vee Y)\to H_i(X\vee Y)$ is an isomorphism for all $i$. Therefore, $g_*\colon H_i(X)\to H_i(X)$ and $h_*\colon H_i(Y)\to H_i(Y)$ is an isomorphism for all $i$, i.e $g\in \Aut(X)$ and $h\in \Aut(Y)$. Hence $$\NAA(X\vee Y) = n \geq \max\{\NAA(X), \NAA(Y)\}.$$

Further, assume that each map in $\AA^m(X\vee Y)$ is $m$-reducible, where $m = \max\{\NAA(X), \NAA(Y)\}$. Let $f\in \AA^m(X\vee Y)$. By $m$-reducibility assumption, it follows that $f_{XX}\in \AA^m(X), f_{YY}\in \AA^m(Y)$. Thus $f_{XX}\in \Aut(X)$ and $f_{YY}\in \Aut(Y)$ since $m = \max \{\NAA(X), \NAA(Y)\}.$  From \cite[Lemma 2.1]{HWSW} we have $(f_X, \iota_Y), (\iota_X, f_Y)\in \Aut(X\vee Y)$, where $f_X = f\circ \iota_X, f_Y = f\circ \iota_Y$. Observe that $$f = (f_X, f_Y) = (\iota_X, f_Y)\circ \big((\iota_X, f_Y)^{-1}\circ f_X, \iota_Y\big).$$
Since $\Aut(X\vee Y)$ is a group and hence so $(\iota_X, f_Y)^{-1}\in \Aut(X\vee Y)$. Therefore, $(\iota_X, f_Y)^{-1}\circ (f_X, \iota_Y)\in \Aut(X\vee Y)\subset \AA^m(X\vee Y)$.
Moreover, observe that $$(\iota_X, f_Y)^{-1}\circ (f_X, \iota_Y) = \big((\iota_X, f_Y)^{-1}\circ f_X, (\iota_X, f_Y)^{-1}\circ \iota_Y \big)\in \AA^m(X\vee Y).$$ By the $m$-reducibility assumption, we have $p_X\circ (\iota_X, f_Y)^{-1}\circ f_X \in \AA^m(X) = \Aut(X)$. It follows that $\big((\iota_X, f_Y)^{-1}\circ f_X, \iota_Y\big)\in \Aut(X\vee Y)$. Thus, $f\in \Aut(X\vee Y)$ and hence so $\NAA(X\vee Y)\leq m = \max \{\NAA(X), \NAA(Y)\}$.

\item The first part directly follows from \cite[Proposition 46]{OYSCFC}.

Further, assume that $h\in \AA^m(X\times Y)$. By the $m$-reducibility assumption, we have $h_{XX}\in \AA^m(X) = \Aut(X)$ and $h_{YY}\in \AA^m(Y) = \Aut(Y)$. From \cite[Proposition 2.3]{PPSHE} we have $(p'_X, h_Y), (h_X, p'_Y)\in \Aut(X\times Y)$, where $h_X = p'_X\circ h$ and $h_Y = p'_Y\circ h$. Observe that $$h = (h_X, h_Y) = \Big(p'_X, h_Y\circ (h_X, p'_Y)^{-1}\Big)\circ (h_X, p'_Y).$$
As in the above process, we have $h\in \Aut(X\times Y)$ and so $\NAA(X\times Y) \leq m = \max\{\NAA(X), \NAA(Y)\}$.
\end{enumerate}
\end{proof}

The following corollary is a direct consequence of Theorem \ref{equality_self-closeness_wedge_product}.
\begin{corollary}\label{cor: redu_product_wedge}
Let $X$ and $Y$ be CW-complexes. The following results hold:
\begin{enumerate}[(a)]
\item $\NAA(X\times Y)\leq \NAA(X\vee Y)$, if each map in $\AA^k(X\times Y)$ is $k$-reducible for $k\geq 1$.

\item $\NAA(X\vee Y)\leq \NAA(X\times Y)$, if each map in $\AA^k(X\vee Y)$ is $k$-reducible for $k\geq 1$. 
\end{enumerate}
\end{corollary}

Combining Corollary \ref{redu_monoid_wedge_product} and Corollary \ref{cor: redu_product_wedge}, we obtain the following result.
\begin{corollary}\label{prop31}
Suppose $X$ and $Y$ are finite-dimensional CW-complexes such that $\conn(X\times Y, X\vee Y)\geq \dim(X\vee Y)$. If every self-map in $\AA^k(X\vee Y)$ is a $k$ -reducible for $k\geq 1$, then $\NAA(X\times Y) = \NAA(X\vee Y) =\max \{ \NAA(X) , \NAA(Y)\}.$
\end{corollary}

The following proposition establishes relations among the homology self-closeness numbers of wedge sums, smash products and products of spaces.
\begin{proposition}\label{wedge_smash_product}
Let $X$ and $Y$ be finite-dimensional CW-complexes. Assume that $\conn(X\times Y, X\vee Y)\geq \dim(X\vee Y)$. Then the following results hold:
\begin{enumerate}[(a)]\setlength\itemsep{1em}
\item $\NAA(X\vee Y) + 1 \leq \NAA(X\wedge Y) \leq \dim(X) + \dim(Y).$
\item $\NAA(X\times Y)\leq \NAA(X\wedge Y)$.
\end{enumerate}
\end{proposition}

\begin{proof}
\begin{enumerate}[(a)]
\item Under the given hypothesis, we have $\pi_i(X\times Y, X\vee Y) = 0$ for all $i\leq \dim(X\vee Y)$. Therefore, $H_i(X\wedge Y) \cong H_i(X\times Y, X\vee Y) = 0$ for all $i\leq \dim(X\vee Y)$, by Hurewicz theorem. It follows that $$\NAA(X\vee Y) + 1\leq \dim(X\vee Y) +1 \leq \NAA(X\wedge Y)\leq \dim(X\wedge Y) = \dim(X) + \dim(Y).$$

\item Assume that $\NAA(X\wedge Y) = m$ and $f\in \AA^m(X\times Y).$ By Lemma \ref{homo_monoid}(ii), we obtain $\bar{f}\in \AA^m(X\wedge Y) = \Aut(X\wedge Y)$. From part (a), we have $\NAA(X\vee Y)\leq m-1$. Applying five lemma to the commutative Diagram in \eqref{diag_five}, we see that $f|_{X\vee Y}\in \AA^{m-1}(X\vee Y) = \Aut(X\vee Y)$. Further, using the five lemma again, it follows that $f\in \Aut(X\times Y)$. This completes the argument.
\end{enumerate}
\end{proof}

The following corollary follows immediately from Proposition \ref{wedge_smash_product}.
\begin{corollary}
Suppose $X$ and $Y$ are finite-dimensional CW-complexes such that $\conn(X\times Y, X\vee Y)\geq \dim(X\vee Y)$. Then $\NAA(X\wedge Y) \geq \max \{\NAA(X), \NAA(Y)\} + 1$.
\end{corollary}

\begin{proposition}
Let $X$ and $Y$ be simply connected CW-complexes. Then 
$$\NAA\big(\Sigma(X\times Y)\big)\geq \NAA\big(\Sigma(X\wedge Y)\big).$$
\end{proposition}

\begin{proof}
From \cite[Proposition 4I.1]{AHAT}, we obtain $\Sigma(X\times Y)\simeq \Sigma X \vee \Sigma Y\vee \Sigma(X\wedge Y)$. Therefore, using Theorem \ref{equality_self-closeness_wedge_product}, we get the desired result. 
\end{proof}

\begin{proposition}[{\cite[Theorem 10]{CHS}}]
Let $X$ be an $(n-1)$ connected space with $\dim(X) \leq 2(n-1)$. Then $\NAA(\Sigma^n X) = \NAA(X) + n$.
\end{proposition}

%\begin{theorem}
%Let $X$ and $Y$ be two CW-complexes. \textcolor{red}{Under some conditions}, we have $$\NAA(X\wedge Y) = \NAA(X) + \NAA(Y).$$
%\end{theorem}

\begin{example}
Consider a wedge sum $\C^m\vee \H^n$. For $k \geq  0$, every self-map $f\in \AA^k(X\vee Y)$ is $k$-reducible, from Example \ref{redu_example}.
Therefore, by Theorem \ref{equality_self-closeness_wedge_product}, we get $$\NAA(\C^m\vee \H^n) = \max \{\NAA(\C^m), \NAA(\H^n)\} = 4.$$
\end{example}

\begin{example}\label{self-closeness number projective spaces}
Consider the monoid $\AA^k(\E^m\vee \E^n)$, where $\mathbb{E}= \mathbb{C} \text{ or } \mathbb{H}$. From Example \ref{redu_example}, we see that every self-map in $\AA^k(\E^m\vee \E^n)$ is $k$-reducible for all $k\geq 0$, whenever $m\neq n$. By Theorem \ref{equality_self-closeness_wedge_product}, we have  the following result: $$\NAA(\E^m\vee \E^n) =  \max\{\NAA(\E^m), \NAA(\E^n)\} = 2 \text{ or }4 \text{ whenever } \mathbb{E} = \mathbb{C} \text { or }\mathbb{H}.$$
\end{example}

\begin{example}
Consider a wedge sum $K(G,2n)\vee \Sigma \mathbb{R}P^{2n}$, where $G$ is an abelian group other than $\ZZ/{2\ZZ}$. Note that $H_{i}(\Sigma \mathbb{R}P^{2n})\cong 0$ for all $i\geq 2n+1$. This implies $\NAA(\Sigma \mathbb{R}P^{2n})\leq 2n$. Moreover, we know that $\NAA(K(G,2n)) = 2n$. By Example \ref{exam:redu_Eilenberg_Real Projective} and Theorem \ref{equality_self-closeness_wedge_product}, we obtain $$\NAA(K(G,2n)\vee \Sigma \mathbb{R}P^{2n}) = \max\big\{\NAA(K(G,2n), \NAA(\Sigma \mathbb{R}P^{2n})\big\} = 2n.$$
\end{example}

\begin{proposition}
Let $G$ and $H$ be two finite groups with no common direct factor. For $m\geq n\geq 2$, if $H_m(K(G,n))$ is finite and contains no direct factor isomorphic to $H$ then $$\NAA(K(G,n)\vee K(H,m)) = m.$$
\end{proposition}

\begin{proof}
Consider $X = K(G,n)$ and $Y = K(H,m)$. We know that $\NAA(X) = \NAA(K(G,n)) = n$ and $\NAA(Y) = \NAA(K(H,m)) = m$. Let $f\in \AA^m(X\vee Y)$. Assume first that $m\neq n$. By Lemma \ref{red}(i), we have $f_{XX}\in \AA^{m-1}(X)$ and $f_{YY}\in \AA^{m-1}(Y)$. Further, using Lemma \ref{red}(iv) it follows that $f_{XX}\in \AA^m(X)$ and $f_{YY}\AA^m(Y)$.

Now consider the case $m=n$. In this situation, Lemma \ref{red}(iv) shows that every map in $\AA^m\big(K(G,n)\vee K(H,m)\big)$ is $m$-reducible. Consequently, by applying Theorem \ref{equality_self-closeness_wedge_product}, we obtain the desired result.
\end{proof}

\subsection{Special Case}
For $m=n$ in Example \ref{self-closeness number projective spaces}, the notion of $k$-reducibility does not arise. In this scenario, the problem simplifies considerably, and computation of the homology self-closeness number does not require any $k$-reducibility assumptions. Consequently, the homology self-closeness number can be determine directly, primarily through the use of the cohomology ring structure.

\begin{proposition}\label{wed}
Consider a wedge of two projective spaces of the same dimension $\E^n \vee \E^n$. Then $\NAA(\E^n\vee \E^n) = 
\begin{cases}
  2 &\text{ if } \mathbb{E} = \mathbb{C},\\
  4 &\text { if } \mathbb{E} = \mathbb{H}.
\end{cases}
$
\end{proposition}

\begin{proof}
First, we consider the case $\mathbb{E} = \mathbb{C}$. The cohomology rings of the complex projective spaces are given by $$H^*(\C^n) = \ZZ[\alpha]/(\alpha^{n+1}), ~~H^*(\C^n) = \ZZ[\beta]/(\beta^{n+1}), ~\text{ where } \deg(\alpha) = \deg(\beta) = 2.$$
Let $f\in \AA^2 (\C^n\vee \C^n)$. Then $f\in \AAA_2(\C^n\vee \C^n)$ by universal coefficient theorem. For the ring homomorphism $f^*\colon H^*(\C^n\vee \C^n)\to H^*(\C^n\vee \C^n)$, we have $$f^*(\alpha) = a \alpha + b \beta, ~~ f^*(\beta) = c \alpha + d \beta, ~~\text{ where } ad-bc = \pm 1.$$ 
Moreover, using the multiplicativity of ring homomorphism, we obtain $$f^*(\alpha \beta) = ac \alpha^2 + (ad+bc) \alpha \beta + bd \beta^2.$$ However, in the cohomology ring $H^*(\C^n\vee \C^n)$, we see that $\alpha \beta = 0$. It follows that $ac = 0, ~~ bd = 0$.
So, either $ad = \pm 1, b = 0, c = 0 ~\text{ or }~ a = 0, d = 0, bc = \pm 1.$ Hence $$f^*(\alpha, \beta) = (\pm \alpha, \pm \beta) ~\text{ or } ~(\pm \beta, \pm \alpha).$$
Therefore, $f^*\colon H^*(\C^n\vee \C^n)\to H^*(\C^n\vee \C^n)$ is a ring isomorphism. Thus $f\in \Aut(\C^n\vee \C^n)$ and hence so $\NAA(\C^n\vee \C^n)\leq 2$. Further, $\C^n \vee \C^n$ is a simply connected CW-complex. This implies $\NAA(\C^n\vee \C^n)\geq 2$ and we get the desired result. 

The proof is similar for the case $\mathbb{E} = \mathbb{H}$.
\end{proof}

\begin{remark}
An element $f\in \A^2(\C^n\vee \C^n)$ implies that $f\in \AA^2(\C^n\vee \C^n)$ by Hurewicz theorem. Hence, $$\NA(\C^n \vee \C^n) = \NAA(\C^n \vee \C^n) = \NAAA(\C^n\vee \C^n) = 2.$$

Similarly for $\H^n\vee \H^n$, we have $$\NA(\H^n \vee \H^n) = \NAA(\H^n \vee \H^n) = \NAAA(\H^n\vee \H^n) = 4.$$
\end{remark}

For $m\neq n$, the (homotopy) self-closeness number for the product of projective spaces $\E^m\times \E^n$ has been proved in \cite[Example 4.5]{LSPS}. Moreover, we know that $$H^*(\E^n\times \E^n)\cong \ZZ[\alpha,\beta]/(\alpha^{n+1},\beta^{n+1}),$$ where $\deg(\alpha) = \deg(\beta) = 2 \text{ or } 4 \text{ whenever } \mathbb{E} = \mathbb{C} \text{ or } \mathbb{H}$. Therefore, using the same argument as in the proof of Proposition \ref{wed}, we obtain the following result.
\begin{proposition}
$\NAA(\E^n\times \E^n) = 
\begin{cases}
  2 &\text{ if } \mathbb{E} = \mathbb{C},\\
  4 &\text { if } \mathbb{E} = \mathbb{H}.
\end{cases}
$ 
\end{proposition}

\begin{proof}
Let $f\in \AA^2(\C^n\times \C^n)$. Arguing analogously to the proof of Proposition \ref{wed}, we obtain
$f^*(\alpha) = a\alpha+b\beta,~ f^*(\beta) = c\alpha+d\beta,~ \text{ where } ad-bc = \pm 1.$ 

\noindent Suppose $A = \begin{bmatrix}
a & b \\
c & d 
\end{bmatrix}.$ 
Then $\det(A) = \pm 1$. Consider a basis $\{\alpha^i\beta^j: i+j = m\}$ of the free $\ZZ$-module $H^{2m}(\C^n\times \C^n)$. Let $M$ be an ${(m+1)\times (m+1)}$ matrix associated to the map $f^*\colon H^{2m}(\C^n\times \C^n)\to H^{2m}(\C^n\times \C^n)$ defined by $$f^*(\alpha^i\beta^j) = (a\alpha+b\beta)^i(c\alpha+d\beta)^j = \Big(\sum^{i}_{r = 0} \binom{i}{r} a^rb^{i-r}\alpha^r\beta^{i-r}\Big)\Big(\sum^j_{l=0} \binom{j}{l} c^ld^{j-l}\alpha^l\beta^{j-l}\Big),$$ where $i+j = m.$ Therefore, $\det(M) = (\det(A))^{\frac{m(m+1)}{2}} = \pm 1$. It follows that $f^*$ is an isomorphism for all $m$. Hence, $f\in \Aut(\C^n\times \C^n)$. The proof for the case $\mathbb{E} = \mathbb{H}$ is analogous. This completes the proof. 
\end{proof}
%---------------------------------------------

\sect{Atomic decomposition of wedge sums}\label{atomic decomposition}
In homotopy theory, the term \emph{atomic space} refers to a space that is "homotopically indecomposable. In other words, an atomic space is one that cannot be further decomposed. In this section, we define \emph{$n$-atomic spaces}, which are used to determine the $k$-reducibility of self-maps in a monoid of wedge sum of such spaces. Here, we adopt the notion of atomic space introduced by Adams and Kuhn in \cite{AKAtom}. First, we recall the definition of an atomic space in our setting.
\begin{definition}[{\cite[Section 6]{PPRSE}}]
A simply connected CW-complex $X$ is called $n$-atomic if, for every self-map $f\colon X\to X$, either
\begin{enumerate}[(i)]
\item $f$ induces an automorphism on $H_i(X)$ for $i\leq n$ or
\item $f$ induces a nilpotent endomorphism on $H_i(X)$ for $i\leq n$.
\end{enumerate} 
Moreover, $X$ is called atomic if $n=\infty$. 
\end{definition}

\begin{example}\label{example_atomic}
For $n\geq 2$, consider an Eilenberg–MacLane space $K(\ZZ/{p^r\ZZ},n)$, where $p$ is a prime and $r\geq 1$. Then $K(\ZZ/{p^r\ZZ},n)$ is an $n$-atomic space, because an endomorphism on the $n$-th homology $f_*\colon \ZZ/{p^r\ZZ}\to \ZZ/{p^r\ZZ}$ is either an automorphism or a nilpotent. Moreover, it is an atomic space. If $l>n$, then any self-map $g\colon K(\ZZ/{p^r\ZZ},n)\to K(\ZZ/{p^r\ZZ},n)$ induces a nilpotent endomorphism on homology up to degree $l$ whenever $g_{\#}\colon \ZZ/{p^r\ZZ}\to \ZZ/{p^r\ZZ}$ is nilpotent, by \cite[Lemma 3.4]{PPRSE}. If $g_{\#}$ is not a nilpotent, then it must be an automorphism. This implies $g_*\colon \ZZ/{p^r\ZZ}\to \ZZ/{p^r\ZZ}$ is an automorphism on $n$-th homology group by Hurewicz theorem. We know that $\NAA\big(K(\ZZ/{p^r\ZZ},n)\big) = n$. It follows that $g\in \AA^n\big(K(\ZZ/{p^r\ZZ},n)\big) = \Aut\big(K(\ZZ/{p^r\ZZ},n)\big)$. Hence, it is an $l$-atomic space and therefore an atomic space.
\end{example}

Similarly, for $n\geq 2$, the Moore space $M(\ZZ/{p^r\ZZ},n)$ is an atomic space.

\begin{example}
Let $X = M(\ZZ/{p^r}\ZZ,n)\vee M(\ZZ/{q^s\ZZ},2n)$, where $p,q$ are primes and $r,s\geq 1, n\geq 2$. Note that $H_n(X) \cong \ZZ/{p^r\ZZ}$ and $H_{2n}(X) \cong \ZZ/{q^s\ZZ}$. For a self-map $f\colon X\to X$, the induced homomorphism $f_*\colon H_i(X)\to H_i(X)$ is either an isomorphism or nilpotent, for $i=n,2n$. However, the fact that $f_*\colon H_{2n}(X)\to H_{2n}(X)$ is an isomorphism does not necessarily imply that $f_*\colon H_n(X)\to H_n(X)$ is an isomorphism. Therefore, $X$ is an $n$-atomic space but is not $2n$-atomic. Indeed $X$ is the wedge sum of two atomic spaces $M(\ZZ/{p^r\ZZ},n)$ and $M(\ZZ/{q^s\ZZ},2n)$.
\end{example}

\begin{example}
Consider a wedge of spaces $K(\ZZ/{p^r\ZZ},n)\vee K(\ZZ/{q^s\ZZ},2n)$, where $p,q$ are primes and $r,s\geq 1, n\geq 2$. Take $X = K(\ZZ/{p^r\ZZ},n)$ and $Y = K(\ZZ/{q^s\ZZ},2n)$. Observe that 
$$
H_i(X\vee Y)\cong 
\begin{cases}
   H_i(X) & \text{ if } n\leq i < 2n,\\ 
   H_i(X)\oplus H_i(Y) & \text{ for } i\geq 2n. 
\end{cases}
$$
Let $f\colon X\vee Y\to X\vee Y$ be a self-map. Fix an integer $n\leq j<2n$, and assume that $f_*\colon H_j(X\vee Y)\to H_j(X\vee Y)$ is an isomorphism. There exists a map $\phi\colon H_j(X\vee Y)\to H_j(X\vee Y)$ such that $M_j(f)\cdot M(\phi) = M(\phi)\cdot M_j(f) = M(\Id_{H_j(X\vee Y)})$. This gives $(f_{XX})_*\phi_{11} = \phi_{11}(f_{XX})_* = \Id_{H_j(X)}$. Thus, $\phi_{11}, (f_{XX})_*\in \Aut(H_j(X))$. Therefore, $f_{XX}\colon X\to X$ induces automorphism on $H_i(X)$ for all $i$, since $X$ is an atomic space. Suppose $\psi^i_{11}\colon H_i(X)\to H_i(X)$ is an inverse of the isomorphism $(f_{XX})_*\colon H_i(X)\to H_i(X)$ for $n\leq i<2n$. This implies that  the endomorphism $f_*\colon H_i(X\vee Y)\to H_i(X\vee Y)$ is an isomorphism with inverse $\psi_i\colon H_i(X\vee Y)\to H_i(X\vee Y)$ for $n\leq i< 2n$, where 
$$
M_i(f)=
\begin{bmatrix}
(f_{XX})_{*} & 0 \\
0 & 0 
\end{bmatrix}
\text{ and }
M(\psi_i)=
\begin{bmatrix}
\psi^i_{11} & 0 \\
0 & 0 
\end{bmatrix}.
$$
Moreover, if $f$ induces a nilpotent endomorphism on $H_i(X\vee Y)$ for some $n\leq i<2n$, then the induced map $f_*\colon H_n(X\vee Y)\to H_n(X\vee Y)$ is nilpotent. Hence, $(f_{XX})_*\colon H_i(X)\to H_i(X)$ is nilpotent for $i=n$ and threfore for all $i$. Consequently, $f$ induces a nilpotent endomorphism on $H_i(X\vee Y)$ for all $n\leq i<2n$. 

Note that any self-map $g\colon X\vee Y\to X\vee Y$ induces either an isomorphism or a nilpotent endomorphism on $H_n(X\vee Y)\cong \ZZ/{p^r\ZZ}$. An isomorphism $g_*\colon H_n(X\vee Y)\to H_n(X\vee Y)$ does not necessarily induce an isomorphism on $H_{2n}(X)\oplus H_{2n}(Y)$. Hence $X\vee Y$ is an $(2n-1)$-atomic space. Further, by Example \ref{example_atomic}, we see that $X$ and $Y$ are both atomic spaces. 
\end{example}

\begin{proposition}\label{atomicity_loop_suspension}
Let $X$ be a simply connected CW-complex. Then 
\begin{enumerate}[(a)]\setlength\itemsep{0.5 em}
\item $X$ is an $n$-atomic space if $\Sigma^m X$ is an $(n+m)$-atomic space for some $m\geq 1$.

Moreover, let $X$ be $r$ connected and $\dim(X)\leq 2r+1$. If $X$ is an $n$-atomic space, then $\Sigma^m X$ is an $(n+m)$-atomic space for all $m\geq 1$.
\item $X$ is an $n$-atomic space if $\Omega^m X$ is an $(n-m)$-atomic space for $n > m\geq 1$.
\end{enumerate}
\end{proposition}

\begin{proof}
\begin{enumerate}[(a)]
\item Note that $\Sigma^m X$ is a $(m+1)$ connected CW-complex. Let $f\colon X\to X$ be a self-map. For $m\geq 1$, by Diagram \eqref{redu_suspension}, we see that $f_*\colon H_i(X)\to H_i(X)$ is either an isomorphism or a nilpotent, for $i\leq n$ if and only if $(\Sigma^m f)_*\colon H_{i+m}(\Sigma^m X)\to H_{i+m}(\Sigma^m X)$ is either an isomorphism or a nilpotent, for $i+m\leq n+m$. It follows that $f\colon X\to X$ induces either an automorphism or a nilpotent endomorphism on homology for $i\leq n$, since $\Sigma^m X$ is an $(n+m)$-atomic space for some $m\geq 1$.

For the converse part, observe that the map $[X,X]\xrightarrow{\Sigma^m} [\Sigma^m X, \Sigma^m X]$ is a surjection, by \cite[Theorem 1.21]{CSH}. For a self-map $g\colon \Sigma^m X\to \Sigma^m X$, we have a map $f'\colon X\to X$ such that $\Sigma^m f' = g$. Thus, the $n$-atomicity of the map $f'$ implies the $(n+m)$-atomicity of the map $g$, for all $m\geq 1$. This completes the proof.

\item Let $f\colon X\to X$ be a self-map. Consider a commutative diagram:
$$
\xymatrixrowsep{2 em}
\xymatrixcolsep{5 em}
\xymatrix{
H_{i+m}(X) \ar[d]_{f_*} & \pi_{i+m}(X) \ar[l]_{h} \ar[r]^{\cong} \ar[d]^{f_\#}  & \pi_i(\Omega^m X) \ar[d]^{(\Omega^m f)_\#} \ar[r]^{h} & H_i(\Omega^m X) \ar[d]^{(\Omega^m f)_*} \\
H_{i+m}(X) & \pi_{i+m}(X) \ar[l]^{h} \ar[r]_{\cong} & \pi_i(\Omega^m X) \ar[r]_{h} &  H_i(\Omega^m X).
}
$$
Observe that $f^r_*\circ h = h\circ f^r_{\#}$ and $(\Omega^m f)^s_*\circ h = h\circ (\Omega^m f)^s_{\#}$ for all $r,s\geq 1$. Further, we consider $\Omega^m X$ is simply connected. Thus, $X$ is $(m+1)$ connected. Therefore, if the endomorphism $(\Omega^m f)_*\colon H_i(\Omega^m X)\to H_i(\Omega^m X)$ is a nilpotent, then $(\Omega^m f)_\#\colon \pi_i(\Omega^m X)\to \pi_i(\Omega^m X)$ is a nilpotent, for $i\leq n-m$ by \cite[Theorem 3.5]{PPRSE}. Thus, $f_\#\colon \pi_{i+m}(X)\to \pi_{i+m}(X)$ is a nilpotent endomorphism for $i+m\leq n$. Hence $f_*\colon H_i(X)\to H_i(X)$ is a nilpotent endomorphism for $i\leq n$. In either case, the endomorphism $(\Omega^m f)_*\colon H_i(\Omega^m X)\to H_i(\Omega^m X)$ is an isomorphism for $i\leq n-m$. It follows that $(\Omega^m f)_\#\colon \pi_i(\Omega^m X)\to \pi_i(\Omega^m X)$ is an isomorphism for $i\leq n-m$ and so $f_\#\colon \pi_{i+m}(X)\to \pi_{i+m}(X)$ is an isomorphism for $i+m\leq n$. Finally, we obtain $f_*\colon H_i(X)\to H_i(X)$ is an isomorphism for $i\leq n$. This completes the proof.
\end{enumerate}
\end{proof}

\begin{remark}
Consider the sequence of iterated suspension maps:
$$[X,X]\xrightarrow{\Sigma} [\Sigma X, \Sigma X]\xrightarrow {\Sigma} [\Sigma^2 X, \Sigma^2 X]\xrightarrow{\Sigma} [\Sigma^3 X, \Sigma^3 X]\xrightarrow{\Sigma}\cdots.$$
For a finite dimensional CW-complex $X$, these homotopy groups stabilize after some finite stage. Suppose $\dim(X) = l$ and $\conn(X) = r$. This implies $\dim(\Sigma^m X) = l+m$ and $\conn(\Sigma^m X) = r+m$. If $l+m\leq 2(r+m)$, i.e, $l\leq 2r+m$, then all the homotopy groups stabilize:
$$[\Sigma^m X, \Sigma^m X]\xrightarrow{\Sigma} [\Sigma^{m+1} X, \Sigma^{m+1} X]\xrightarrow{\Sigma} [\Sigma^{m+2} X, \Sigma^{m+2} X]\xrightarrow{\Sigma} \cdots.$$ 
Therefore, given any finite dimensional CW-complex $X$, all the homotopy groups are eventually isomorphic, even without any connectivity assumption on $X$ itself. It follows that for a sufficiently large $m$, if $\Sigma^m X$ is an $n$-atomic space, then $\Sigma^{m+t} X$ is an $(n+t)$-atomic space for all $t\geq 1$.
\end{remark}

Recall that $X$ is said to be dominated by $Y$ (or $X$ is a homotopy retract of $Y$) if there are maps $f\colon X\to Y$ and $g\colon Y\to X$ such that $g\circ f\simeq \Id_X$.

\begin{lemma}\label{atomic_homodis}
\begin{enumerate}[(a)]
\item If $X$ and $Y$ are homologically $n$-distant, then $X$ is not dominated by $Y$.
\item Let $X$ be an $n$-atomic space such that $\NAA(X)\leq n$. If $X$ is not dominated by $Y$, then $X$ and $Y$ are homologically $n$-distant.
\item Let $X$ be an $n$-atomic space such that $\NAA(X)\leq n$. If $X$ is dominated by $Y\vee Z$, then $X$ is dominated by $Y \text{ or } Z$.
\end{enumerate}
\end{lemma}

\begin{proof}
\begin{enumerate}[(a)]
\item If possible, let $X$ be dominated by $Y$. There exist two maps $f\colon X\to Y$ and $g\colon Y\to X$ such that $g\circ f\simeq \Id_X$. By the given assumption $g\circ f\colon X\to X$ induces nilpotent endomorphism on $H_i(X)$ for $i\leq n$. This implies that $\Id_X$ induces nilpotent endomorphism, which is a contradiction.

\item Suppose $X$ and $Y$ are not homologically $n$-distant. Then there exist two maps say $\phi\colon X\to Y$ and $\psi\colon Y\to X$ such that $\psi\circ \phi\colon X\to X$ does not induces a nilpotent endomorphism on $H_k(X)$ for some $k\leq n$. It follows that $\psi\circ \phi$ induces an automorphism on $H_i(X)$ for $i\leq n$. Thus, $\psi\circ \phi\in \Aut(X)$. Therefore, there exists a map $h\colon X\to X$ such that $h\circ (\psi\circ \phi)\simeq \Id_X$ and so $(h\circ \psi)\circ \phi\simeq \Id_X$. This implies $X$ is dominated by $Y,$ a contradiction.

\item Let $f\colon X\to Y\vee Z$ and $g\colon Y\vee Z\to X$ be two maps such that $g\circ f\simeq \Id_X$, since $X$ is dominated by $Y\vee Z$. As in the proof of Lemma \ref{red}, it follows that $(g\circ \iota_Y)_* (p_Y\circ f)_* + (g\circ \iota_Z)_* (p_Z\circ f)_* = \Id_{H_i(X)}$ for all $i$. If $(g\circ \iota_Y)_* (p_Y\circ f)_*\colon H_i(X)\to H_i(X)$ is an automorphism for $i\leq n$, then $(g\circ \iota_Y)\circ (p_Y\circ f)\in \Aut(X)$. By the similar process in part (b), we see that $X$ is dominated by $Y$. Otherwise, $(g\circ \iota_Y)\circ (p_Y\circ f)$ induces a nilpotent endomorphism for $i\leq n$ as $X$ is an $n$-atomic sapce. From the subsection \ref{nilpotent_reducibility}, we obtain $(g\circ \iota_Y)\circ (p_Y\circ f)$ is $n$-quasi-regular. Thus $\Id_{H_i(X)} - (g\circ \iota_Y)_* (p_Y\circ f)_* = (g\circ \iota_Z)_* (p_Z\circ f)_*\colon H_i(X)\to H_i(X)$ is an automorphism for $i\leq n$. Therefore, $(g\circ \iota_Z)\circ (p_Z\circ f)\in \Aut(X)$ and $X$ is dominated by $Z$.
\end{enumerate}
\end{proof}

\begin{corollary}
Assume that $X$ is not dominated by $Y$, where the ring $\End(H_i(Y))$ is commutative for $i\leq n$. If $X$ is an $n$-atomic space and $\NAA(X)\leq n$, then a self-map $f\in \AA^n(X\vee Y)$ is $n$-reducible if and only if $f_{XX}\in \AA^n(X)$.
\end{corollary}

\begin{proposition}
Let $K(\ZZ/{p^r\ZZ},n)$ not be dominated by $Y$, where $p$ is a prime, $n\geq 2$ and $r\geq 1$. Then any self-map in $\AA^n\big(K(\ZZ/{p^r\ZZ},n)\vee Y\big)$ is $n$-reducible.
\end{proposition}
\begin{proof}
By Lemma \ref{atomic_homodis}(b), $K(\ZZ/{p^r\ZZ},n)$ and $Y$ are homologically $n$-distant. Rest of the proof is similar to that of Proposition \ref{nilpotent_reducibility_homdis}.
\end{proof}

The following theorem establishes a sufficient condition of $k$-reducibility of self-maps in a monoid of two atomic spaces.
\begin{theorem}\label{reducibility_atomic}
Let $X$ and $Y$ be $m$-atomic and $n$-atomic spaces, respectively. Fix an integer $i_0\leq \min\{m,n\}$ such that $H_{i_0}(X)$ and $H_{i_0}(Y)$ are nontrivial groups. If every endomorphism of $H_{i_0}(X)$ that factors through $H_{i_0}(Y)$ is nilpotent, then each self-map in $\AA^k(X\vee Y)$ is $k$-reducible, where $i_0\leq k\leq \min\{m,n\}$.
\end{theorem}

\begin{proof}
Let $f\in \AA^k(X\vee Y)$, where $k\leq \min\{m,n\}$. Then there exists an isomorphism $\phi_i\colon H_i(X)\oplus H_i(Y)\to H_i(X)\oplus H_i(Y)$ such that $M_i(f)\cdot M(\phi_i) = M(\Id_{H_i(X)\oplus H_i(Y)})$ for $i\leq k$. This implies $$(f_{XX})_*\phi^i_{11} + (f_{XY})_*\phi^i_{21} = \Id_{H_i(X)} \text{ and } (f_{YX})_*\phi^i_{12} + (f_{YY})_*\phi^i_{22} = \Id_{H_i(Y)}.$$ Let us take $i_0\leq \min\{m,n\}$ such that $H_{i_0}(X)$ and $H_{i_0}(Y)$ are nontrivial groups. Moreover, each endomorphism of $H_{i_0}(X)$ that factors through $H_{i_0}(Y)$ is nilpotent. From symmetry, every endomorphism of $H_{i_0}(Y)$ that factors through $H_{i_0}(X)$ is also nilpotent. Therefore, $(f_{XX})_*\phi^i_{11} = \Id_{H_{i_0}(X)} - (f_{XY})_*\phi^i_{21}\in \Aut\big(H_{i_0}(X)\big)$ and $(f_{YY})_*\phi^i_{22} = \Id_{H_{i_0}(Y)} - (f_{YX})_*\phi^i_{12}\in \Aut\big(H_{i_0}(Y)\big).$ This implies both the endomorphisms $(f_{XX})_*\colon H_{i_0}(X)\to H_{i_0}(X)$ and $(f_{YY})_*\colon H_{i_0}(Y)\to H_{i_0}(Y)$ are not nilpotent. Hence $(f_{XX})_*\in \Aut\big(H_i(X)\big)$ and $(f_{YY})_*\in \Aut\big(H_i(Y)\big)$ for all $i\leq k$, since $X$ is $m$-atomic and $Y$ is $n$-atomic. This completes the proof.
\end{proof}

The following result follows immediately from the proof of Theorem \ref{reducibility_atomic}.
\begin{corollary}
Let $X$ and $Y$ be $m$-atomic and $n$-atomic spaces, respectively. Fix an integer $i_0\leq \min\{m,n\}$ such that $H_{i_0}(X)$ and $H_{i_0}(Y)$ are nontrivial groups. If each map in $\AA^{i_0}(X\vee Y)$ is $i_0$-reducible, then all maps in $\AA^k(X\vee Y)$ are $k$-reducible, for all $i_0\leq k\leq \min\{m,n\}.$
\end{corollary}

\begin{example}\label{example_Eilenberg_Maclane_redu}
Consider two atomic spaces $K(\ZZ/{p^r\ZZ},m)$ and $K(\ZZ/{q^s\ZZ},n)$, where $p,q$ are primes not necessarily distinct, $r,s\geq 1$ and $m,n\geq 2$.  We obtain the following observations:

\noindent \textbf{Case-I}: Suppose that $m = n, ~ p = q$ and $r\neq s$. For $k = n$, each self-map in $\AA^k\big(K(\ZZ/{p^r\ZZ},m)\vee K(\ZZ/{q^s\ZZ},n)\big)$ is $k$-reducible by Theorem \ref{reducibility_atomic}. Moreover, for $k<n$, there is no homology groups for both the Eilenberg-MacLane spaces. Now consider the case $k>n$. Let $f\in \AA^k\big(K(\ZZ/{p^r\ZZ},m)\vee K(\ZZ/{q^s\ZZ},n)\big)\subseteq \AA^n\big(K(\ZZ/{p^r\ZZ},m)\vee K(\ZZ/{q^s\ZZ},n)\big)$. Therefore, $f_{XX}\in \AA^n(K(\ZZ/{p^r\ZZ},m)) = \Aut(K(\ZZ/{p^r\ZZ},n))$ and $f_{YY}\in \AA^n(K(\ZZ/{q^s\ZZ},n)) = \Aut(K(\ZZ/{q^s\ZZ},m))$. Consequently, for $k\geq 1$, every self-map in $\AA^k\big(K(\ZZ/{p^r\ZZ},m)\vee K(\ZZ/{q^s\ZZ},n)\big)$ is $k$-reducible.
\vspace{0.5 em}

\noindent \textbf{Case-II}: Let $m = n$ and $p\neq q$. Thus, any self-map in $\AA^k\big(K(\ZZ/{p^r\ZZ},m)\vee K(\ZZ/{q^s\ZZ},n)\big)$ is $k$-reducible similar to Case-I.
\vspace{0.5 em}

\noindent \textbf{Case-III}: Assume $m\neq n$. Without loss of generality, take $m>n$. Each map in $\AA^k\big(K(\ZZ/{p^r\ZZ},m)\vee K(\ZZ/{q^s\ZZ},n)\big)$ is $k$-reducible for $k<m$ by Lemma \ref{red}(i). Further, $K(\ZZ/{p^r\ZZ},m)$ is not dominated by $K(\ZZ/{q^s\ZZ},n)$ and vice versa. Thus, $K(\ZZ/{p^r\ZZ},m)$ and $K(\ZZ/{q^s\ZZ},n)$ are homologically $m$-distant by Lemma \ref{atomic_homodis}(ii). Let $f\in \AA^m\big(K(\ZZ/{p^r\ZZ},m)\vee K(\ZZ/{q^s\ZZ},n)\big)\subseteq \AA^{m-1}\big(K(\ZZ/{p^r\ZZ},m)\vee K(\ZZ/{q^s\ZZ},n)\big)$. Then $f_{YY}\in \AA^{m-1}\big(K(\ZZ/{q^s\ZZ},n)\big) = \Aut\big(K(\ZZ/{q^s\ZZ},n)\big)$. Therefore, using Proposition \ref{prop_redu_end}, we obtain the reducibility. This implies every map in $\AA^k\big(K(\ZZ/{p^r\ZZ},m)\vee K(\ZZ/{q^s\ZZ},n)\big)$ is $k$-reducible for $k\geq 1$, since both the spaces are atomic. 
\end{example}

%------------------------------------------------

\sect{Application}\label{application}
In this section, we discuss the $k$-reducibility of self-maps in the monoids of atomic spaces and its related homology self-closeness numbers. By applying various results from the previous sections, we compute the homology self-closeness number of wedge sums of spaces. 

The following result establishes relations between the $n$-atomicity and the homology self-closeness number.
\begin{lemma}\label{atomic_selfcloseness}
Let $X$ be a $n$-atomic space such that $\NAA(X)\leq n$. Then
\begin{enumerate}[(a)]
\item  If $H_k(X)\neq 0$ for some integer $2\leq k\leq n$, then $\NAA(X)\leq k$.
\item If $\NAA(X) = m$, then $X$ is a $(m-1)$ connected CW-complex.
\item $\NAA(X) = \min\{i: H_i(X)\neq 0\}.$
\end{enumerate}
\end{lemma}

\begin{proof}
\begin{enumerate}[(a)]
\item Suppose $H_k(X)\neq 0$ for some $2\leq k\leq n$ and $f\in \AA^k(X)$. From the definition of $n$-atomic space, we have $f_*\colon H_i(X)\to H_i(X)$ is an isomorphism for all $i\leq n$. Hence $f\in Aut(X)$ and we get the desired result.
\item If $X$ is not $(m-1)$ connected, then there exists an integer $k$  such that $H_k(X)\neq 0$, where $2\leq k \leq m-1$. Thus, by part (a), we have $\NAA(X)\leq k\leq m-1$, which contradicts the given assumption.
\item Follows from part (a) and (b).
\end{enumerate}
\end{proof}

The following corollary follows directly from Lemma \ref{atomic_selfcloseness}:
\begin{corollary}
Let $X$ be an atomic space. Then $$\NAA(X) = \min\{i: H_i(X)\neq 0\}.$$  
\end{corollary}

For a Moore space, by definition we have $\NAA(M(G,n) = n$. Moreover, $M(G,n)\vee M(H,m) = M(G\oplus H,n)$, whenever $m=n$. In addition, Lemma \ref{red}(i) and Theorem\ref{equality_self-closeness_wedge_product} can be extended to a finite wedge of spaces $X_1,\ldots,X_l$, provided that $Hom\big(H_i(X_j), H_i(X_{j'})\big) =0$ for all $1\leq j\neq j'\leq l$. Therefore, the following result can be easily deduced.
\begin{example}[{\cite[Example 5.1]{LSPS}}]
Suppose $n_1,\ldots,n_l\geq 2$ and $G_1,\ldots,G_l$ are abelian groups. Then $\NAA\big(M(G_1,n_1)\vee \cdots \vee M(G_l,n_l\big) = \max\{n_i: i=1,\ldots,l\}.$
\end{example}

In the following theorem, we determine the $k$-reducibility of self-maps on the wedge of two atomic spaces and the associated homology self-closeness number.
\begin{theorem}\label{thm_selfcloseness_redu}
Let $X_1$ and $X_2$ be $n_1$ and $n_2$-atomic spaces, respectively, such that $\NAA(X_1) \leq n_1$ and $\NAA(X_2)\leq n_2$. If $\NAA(X_1)\neq \NAA(X_2)$, then every self-map in $\AA^m(X_1\vee X_2)$ is $m$-reducible, where $m=\max\{\NAA(X_1), \NAA(X_2)\}$. Moreover, $$\NAA(X_1\vee X_2) = \max\{\NAA(X_1), \NAA(X_2)\}.$$
\end{theorem}

\begin{proof}
Suppose $\NAA(X_1) = m_1$ and $\NAA(X_2) = m_2$. Without loss of generality, assume that $m_2 > m_1$. From Lemma \ref{atomic_selfcloseness}, we see that $X_1$ and $X_2$ are $(m_1-1)$ and $(m_2-1)$ connected spaces, respectively. Therefore, $X_1$ and $X_2$ are homologically $m_1$-distant and no one is dominated by other one, by Lemma \ref{atomic_homodis}(a). Let $f\in \AA^{m_2}(X_1\vee X_2)$. Then similar to the proof of Theorem \ref{reducibility_atomic}, we have $(f_{X_1X_1})_*\phi^{m_1}_{11} = \Id_{H_{m_1}(X_1)} - (f_{X_1X_2})_*\phi^{m_1}_{21} = \Id_{H_{m_1}}(X_1)$ and $\phi^{m_1}_{11}(f_{X_1X_1})_* = \Id_{H_{m_1}(X_1)} - \phi^{m_1}_{12}(f_{X_2X_1})_* = \Id_{H_{m_1}(X_1)}$. Thus, $f_{X_1X_1}\in \Aut(H_{m_1}(X_1))$. This implies $f_{X_1X_1}\in \AA^{m_1}(X_1) = \Aut(X_1)$, since $X_1$ is an $n_1$-atomic space and $m_1\leq n_1$. Further, $X_1$ and $X_2$ are homologically $m_2$-distance by using Lemma \ref{atomic_homodis}(b). Note that $\End(H_i(X_2))$ is commutative for all $i\leq m_2$. Therefore, $f\in \AA^{m_2}(X_1\vee X_2)$ is $m_2$-reducible, from Proposition \ref{prop_redu_end}

Moreover, using Theorem \ref{equality_self-closeness_wedge_product}, we get the desired result.
\end{proof}

We generalize the Theorem \ref{thm_selfcloseness_redu} for the wedges of multiple atomic spaces.
\begin{theorem}\label{thm_higher_selfcloseness_redu}
Let $X_i$ be $n_i$-atomic space such that $\NAA(X_i)\leq n_i$ for $i=1,\ldots,l$. If $\NAA(X_i)\neq \NAA(X_j)$ for $1\leq i\neq j\leq l$, then each self-map in $\AA^m(X_1\vee \cdots \vee X_l)$ is $m$-reducible, where $m = \max\{\NAA(X_i): i =1,\dots, l\}$. Moreover,
$$\NAA(X_1\vee \cdots \vee X_l) = \max\{\NAA(X_i): i =1,\dots, l\}.$$
\end{theorem}

\begin{proof}
We prove the result inductively. For $l=2$, we get the result in Theorem \ref{thm_selfcloseness_redu}. Suppose $\NAA(X_i) = m_i$ for $i=1,2,3$, with $m_3 > m_2 > m_1$. Then, $\conn(X_i) = m_i-1$ for $i=1,2,3$. Therefore, none of $X_i$ is dominated by $X_j$ for $1\leq i\neq j\leq 3$ by \ref{atomic_homodis}(a). From Lemma \ref{atomic_homodis}(c), we have $X_3$ is not dominated by $X_1\vee X_2$. Take $X = X_1\vee X_2$. By Lemma \ref{atomic_homodis}(b), we see that $X$ and $X_3$ are homologically $m_3$-distant. Let $f\in \AA^{m_3}(X\vee X_3)$. As in the proof of Theorem \ref{thm_selfcloseness_redu}, we have $f_{XX}\in \AA^{m_2}(X) = \Aut(X)$. Note that $\End\big(H_i(X_3)\big)$ is commutative for each $i\leq m_3$. Thus, we get the reducibility of $f$ in $\AA^{m_3}(X\vee X_3)$ and so in $\AA^{m_3}(X_1\vee X_2\vee X_3)$. Hence, $$\NAA(X_1\vee X_2\vee X_3) = \max\{\NAA(X),\NAA(X_3)\} = \max\{\NAA(X_i): i=1,2,3\}.$$
For $l\geq 4$, the proof is similar to the case $l=3$, by assuming $X=X_1\vee \cdots \vee X_{l-1}$ and $m_l>\cdots >m_1$, where $\NAA(X_i) = m_i$. This completes the proof.
\end{proof}

\begin{corollary}\label{higher_selfcloseness}
Suppose that either the integers $r_i$ or the integers $n_i$ are mutually distinct for $i=1,\ldots,l$, then $$\NAA\big(K(\ZZ/{p^{r_1}_1\ZZ},n_1)\vee \cdots \vee K(\ZZ/{p^{r_l}_l\ZZ},n_l)\big) = \max\{n_i: i=1,\ldots,l\},$$ where the primes $p_i$ are not necessarily distinct.
\end{corollary}

\begin{proof}
We establish the result by considering the following two cases.

\noindent \textbf{Case-I:} Let $r_i\neq r_j$ for $1\leq i\neq j\leq l$. Thus, the space $K(\ZZ/{p^{r_i}\ZZ}, n_i)$ is not dominated by $K(\ZZ/{p^{r_j}\ZZ}, n_j)$, and vice versa for $i\neq j$. Hence, by an argument identical to that used in the proof of Theorem \ref{thm_higher_selfcloseness_redu}, we get the desired result.

\noindent \textbf{Case-II:} Let $n_i \neq n_j$ for $1\leq i\neq j\leq l$. It is well known that $\NAA\big(K(\ZZ/{p^{r_i}_i\ZZ},n_i)\big) = n_i$ for all $1\leq i\leq l$. Therefore, the result follows directly from Theorem \ref{thm_higher_selfcloseness_redu}. This completes the proof.
\end{proof}

\bibliographystyle{plain} 
\bibliography{references}

\end{document}